%% file: kTdfArXivVersion.tex
\documentclass[english,11pt,reqno]{article}

\usepackage[utf8]{inputenc}
\usepackage[english]{babel}

\usepackage{mathrsfs}
\usepackage{amsmath,amsfonts,amssymb}
\usepackage{bbm}
\usepackage{amsthm}

\usepackage{enumerate,setspace}

\usepackage{fullpage}

\usepackage{color}

\usepackage{natbib}

\usepackage{multirow}

\usepackage{epsfig}
\usepackage{subfigure}

\usepackage{graphicx}

\usepackage{mathtools}
\mathtoolsset{showonlyrefs}
\usepackage{fixltx2e}
\MakeRobust{\eqref}

\newtheorem{theorem}{\bf Theorem }

\newtheorem{definition}[theorem]{\bf Definition}
\newtheorem{lemma}[theorem]{\bf Lemma}
\newtheorem{corollary}[theorem]{\bf Corollary}

 \newcounter{todocounter}
 \newcounter{commentcounter}
\newcommand{\N}{\ensuremath{\mathbb{N}}}

\newcommand{\R}{\ensuremath{\mathbb{R}}}

\renewcommand{\P}{\ensuremath{\mathbb{P}}}

\newcommand{\1}{\ensuremath{\mathbbm{1}}}
\newcommand{\M}{\ensuremath{\mathcal{M}}}

\input{macros.tex}

\begin{document}

\title{An explicit transition density expansion for a multi-allelic Wright-Fisher diffusion with general diploid selection}
\author{Matthias Steinr\"ucken$^{1,*}$, Y. X. Rachel Wang$^{1,\dagger}$ and Yun S. Song$^{1,2,\ddagger}$}

\footnotetext[1]{Department of Statistics, University of California, Berkeley, CA 94720, USA}
\footnotetext[2]{Computer Science Division, University of California, Berkeley, CA 94720, USA}

	%

\date{}

\maketitle

\begin{abstract}
Characterizing time-evolution of allele frequencies in a population is a fundamental problem in population genetics. In the Wright-Fisher diffusion, such dynamics is captured by the transition density function, which satisfies well-known partial differential equations. For a multi-allelic model with general diploid selection, various theoretical results exist on representations of the transition density, but finding an explicit formula has remained a difficult problem. In this paper, a technique recently developed for a diallelic model is extended to find an explicit transition density for an arbitrary number of alleles, under a general diploid selection model with recurrent parent-independent mutation. Specifically, the method finds the eigenvalues and eigenfunctions of the generator associated with the multi-allelic diffusion, thus yielding an accurate spectral representation of the transition density. Furthermore, this approach allows for efficient, accurate computation of various other quantities of interest, including the normalizing constant of the stationary distribution and the rate of convergence to this distribution.
\renewcommand{\thefootnote}{}\footnote{Email: ${}^*${\tt steinrue@stat.berkeley.edu}, ${}^\dagger${\tt rachelwang@stat.berkeley.edu}, ${}^\ddagger${\tt yss@cs.berkeley.edu}}
\end{abstract}


\section{Introduction}
\label{sec_introduction}

Diffusion processes can be used to describe the evolution of population-wide allele frequencies in large populations, and they have been successfully applied in various population genetic analyses in the past. \citet{Karlin1981}, \citet{Ewens2004}, and \citet{Durrett2008} provide excellent introduction to the subject.  The diffusion approximation captures the key features of the underlying evolutionary model and provides a concise framework for describing the dynamics of allele frequencies, even in complex evolutionary scenarios.  However, finding explicit expressions for the transition density function (TDF) is a challenging problem for most models of interest.  Although a partial differential equation (PDE) satisfied by the TDF can be readily obtained from the standard diffusion theory, few models admit analytic solutions.

Since closed-form transition density functions are unknown for general diffusions, approaches such as finite difference methods \citep{Bollback2008,Gutenkunst2009} and series expansions \citep{Lukic2011} have been adopted recently to obtain approximate solutions.
In a finite difference scheme, one needs to discretize the state space, but since the TDF depends on the parameters of the model (e.g. the selection coefficients), the suitability of a given discretization might depend strongly on the parameter values, and whether a particular discretization would produce accurate solutions is difficult to predict \emph{a priori}.
Series expansions allow one to circumvent the problem of choosing an appropriate discretization for the state space.  However, if the chosen basis functions in the representation are not the eigenfunctions of the diffusion generator, as in \citet{Lukic2011}, then one has to solve a system of coupled ordinary differential equations (ODE) to obtain the transition density. \citet{Lukic2011} solve this system of ODEs numerically, which may introduce potential errors of numerical approximations.

If the eigenvalues and eigenfunctions of the diffusion generator can be found, the spectral representation (which in some sense provides the optimal series expansion) of the TDF can be obtained.
For the one-locus Wright-Fisher diffusion with an arbitrary number  $\numAlleles$ of alleles evolving under a \emph{neutral} parent-independent mutation (PIM) model, \citet{Shimakura1977} and \citet{Griffiths1979} derived an explicit spectral representation of the TDF using orthogonal polynomials.  More recently, \citet{Baxter2007} derived the same solution by diagonalizing the associated PDE using a suitable coordinate transformation, followed by solving for each dimension independently.
In a related line of research, \citet{Griffiths1983} and \citet{Tavare1984} expressed the time-evolution of the allele frequencies in terms of a stochastic process dual to the diffusion, and showed that the resulting expression is closely related to the spectral representation.

The duality approach was later extended by \citet{Barbour2000} to incorporate a general selection model.
Although theoretically very interesting, this approach does not readily lead to efficient computation of the TDF because of the following reason:  Computation under the dual process requires evaluating the moments of the stationary distribution. 
Although the functional form of this distribution is known \citep{Ethier1994,Barbour2000}, the normalization constant and moments can only be computed analytically in special cases \citep{Genz2003}, and numerical computation under a general model of diploid selection is difficult.  Incidentally, this issue arises in various applications (e.g., \citealt{Buzbas2009a, Buzbas2011}), and it has therefore received significant attention in the past; see, for example, \citet{Donnelly2001} and \citet{Buzbas2009}.

Many decades ago, \citet{Kimura1955,Kimura1957} addressed the problem of finding an \emph{explicit} spectral representation of the TDF for models with selection.   Specifically, in the case of a \emph{diallelic} model with special selection schemes, he employed a perturbation method to find the required eigenvalues and eigenfunctions of the diffusion generator.  Being perturbative in the selection coefficient, this approach is accurate only for small selection parameters. 
Recently, \citet{Song2012} revisited this problem and developed an alternative method of deriving an explicit spectral representation of the TDF for the diallelic Wright-Fisher diffusion under a \emph{general} diploid selection model with recurrent mutation.
In contrast to \citeauthor{Kimura1955}'s approach, this new approach is non-perturbative and is applicable to a broad range of parameter values. 
The goal of the present paper is to extend the work of \citet{Song2012} to an arbitrary number $\numAlleles$ of alleles, assuming a PIM model with general diploid selection.   

The rest of this paper is organized as follows.  In \sref{sec_background}, we lay out the necessary mathematical background and review the work of \citet{Song2012} in the case of a diallelic ($\numAlleles=2$) model with general diploid selection.  In \sref{sec_neutral_case}, we describe the spectral representation for the \emph{neutral} PIM model with an arbitrary number $\numAlleles$ of alleles. Then, in \sref{sec:selected_case}, we generalize the method of \citet{Song2012} to 
an arbitrary $\numAlleles$-allelic PIM model with general diploid selection.
We demonstrate in \sref{sec_simulations} that the quantities involved in the spectral representation converge rapidly.
Further, we discuss the computation of the normalization constant of the stationary distribution under mutation-selection balance and the rate of convergence to this distribution.  We conclude in \sref{sec_discussion} with potential applications and extensions of our work.

\section{Background}
\label{sec_background}

\subsection{The Wright-Fisher diffusion} 
\label{sec:WF}
In this paper, we consider a single locus with $\numAlleles$ distinct possible alleles.        
The dynamics of the allele frequencies in a large population is commonly approximated by the Wright-Fisher diffusion on the 
$(\numAlleles-1$)-simplex
\begin{equation}
	\simplex_{\numAlleles-1} := \left\{ \xVec  \in \R_{\geq 0}^{\numAlleles-1} : 1 - |{\xVec}| \geq 0 \right\},
\end{equation}
where $|{\xVec}|=\sum_{i=1}^{\numAlleles-1} \xCmp_i$.
For a given ${\xVec}=(\xCmp_1,\ldots,\xCmp_{\numAlleles-1}) \in \simplex_{\numAlleles-1}$, the component $\xCmp_i$ denotes the population frequency of allele $i \in \{1,\ldots,\numAlleles-1\}$.  The frequency of allele $\numAlleles$ is given by $x_\numAlleles = 1 - |\xVec|$.

The associated diffusion generator $\genFull$ is a second order differential operator of form
\begin{equation}\label{eq_backward_generator}
	\begin{split}
		\genFull f({\xVec}) = \frac{1}{2} \sum_{i,j=1}^{\numAlleles-1} \diffTerm_{i,j}({\xVec}) \frac{\partial^2}{\partial \xCmp_i \partial \xCmp_j} f({\xVec}) + \sum_{i=1}^{\numAlleles-1} \driftTerm_i({\xVec}) \frac{\partial}{\partial \xCmp_i}f({\xVec}),
	\end{split}
\end{equation}
which acts on twice continuously differentiable functions $f \colon \simplex_{\numAlleles-1} \to \R$.
The diffusion coefficient $\diffTerm_{i,j}({\xVec})$ is given by
\begin{equation}
	\diffTerm_{i,j}({\xVec}) = \xCmp_i (\delta_{i,j} - \xCmp_j),
\end{equation}
where the Kronecker delta $\delta_{i,j}$ is equal to 1 if $i=j$ and 0 otherwise.
For a neutral PIM model, we use $\mutCmp_i = 4 \popSize \uCoeff_i$ to denote the population-scaled mutation rate associated with allele $i$, where $\uCoeff_i$ is the probability of mutation producing allele $i$ per individual per generation and $\popSize$ is the effective population size. 
Under this model, the drift coefficient $\driftTerm_i({\xVec})$ is given by
\begin{equation}
	\driftTerm_i({\xVec}) = \frac{1}{2} (\mutCmp_i - |\mutVec| \xCmp_i),
\end{equation}
where $\mutVec = (\mutCmp_1,\ldots,\mutCmp_\numAlleles)\in \R_{>0}^\numAlleles$ and $|\mutVec| = \sum_{i=1}^\numAlleles \mutCmp_i$.

Consider a general diploid selection model in which the relative fitness of a diploid individual with one copy of allele $i$ and one copy of allele $j$ is given by  $1 + 2s_{i,j}$.  We measure fitness relative to that of an individual with two copies of allele $K$, thus $s_{K,K} = 0$. The diffusion generator in this case is given by $\genFull  = \genNeutral + \genSel$, where
$\genNeutral$ denotes the diffusion generator under neutrality and the additional term $\genSel$, which 
captures the contribution from selection to the drift coefficient $\driftTerm_i({\xVec})$, is given by 
\begin{equation}\label{eq_drift_selection}
	\genSel= 
	 \sum_{i=1}^{\numAlleles-1}\xCmp_i \left[\selCmp_{i}(\xVec) - \meanFitness({\xVec})\right] \frac{\partial}{\partial \xCmp_i},
\end{equation}
where $\selCmp_{i}(\xVec)$ denotes the marginal fitness of type $i$ and  $\meanFitness (\xVec)$ denotes the mean fitness of the population with allele frequencies $\xVec$.  More precisely, 
\begin{equation}
 \selCmp_{i}(\xVec)=\sum_{j=1}^\numAlleles \selCmp_{i,j} \xCmp_j
\label{eq:marginalFitness}
\end{equation}
and
\begin{equation}
\label{eq:meanFitness}
	\meanFitness (\xVec) := \sum_{i,j=1}^\numAlleles \selCmp_{i,j} \xCmp_i \xCmp_j,
\end{equation}
where $\selCmp_{i,j} = 2 \popSize s_{i,j}$.
Intuitively, the population frequency of a given allele tends to increase if its marginal fitness is higher than the mean fitness of the population. 
The selection scheme is specified by a symmetric matrix $\selMatrix = (\selCmp_{i,j})_{1 \leq i,j \leq \numAlleles} \in \R^{\numAlleles\times \numAlleles}$ of population-scaled selection coefficients, with $\selCmp_{\numAlleles,\numAlleles} = 0$.

The operators $\genNeutral$ and $\genFull$ are elliptic inside the simplex $\simplex_{\numAlleles-1}$, but not on the boundaries. Thus, the precise domain of the generator is not straightforward to describe, but \citet{Epstein2011} give a suitable characterization.
 
\subsection{Spectral representation of the transition density function}
\label{sec:background_spec_rep}
For $t\geq 0$, the time evolution of a diffusion $\alleleFreqRV_t$ on the simplex $\simplex_{\numAlleles-1}$ is described by the transition density function $p(t;\xVec,\yVec)d\yVec = \P [\alleleFreqRV_t \in d\yVec \mid \alleleFreqRV_0 = \xVec ]$, where $\xVec, \yVec \in \simplex_{\numAlleles-1}$.
The transition density function satisfies the Kolmogorov backward equation
\begin{equation}\label{eq_backward_equation}
	\frac{\partial}{\partial t} p(t;\xVec,\yVec) = \genFull p(t;\xVec,\yVec),
\end{equation}
where $\genFull$, the generator associated with the diffusion, is a differential operator in $\xVec$.

We briefly review the framework underlying the spectral representation of the transition density function. The operator $\genFull$ is said to be symmetric
with respect to a density $\pi \colon \simplex_{\numAlleles-1} \to \R_{\geq 0}$ if, for all twice continuously differentiable functions $f \colon \simplex_{\numAlleles-1} \to \R$ and $g \colon \simplex_{\numAlleles-1} \to \R$ that belong to the domain of the operator, the following equality holds:
\begin{equation}
	\int_{\simplex_{\numAlleles-1}} [\genFull f(\xVec)] g(\xVec) \pi(\xVec) d\xVec = \int_{\simplex_{\numAlleles-1}} f(\xVec)  [\genFull g(\xVec)] \pi(\xVec) d\xVec.
\end{equation}
A straightforward calculation using integration by parts yields that the diffusion generators described in \sref{sec:WF} are symmetric with respect to their associated stationary densities.

Theorem~1.4.4 of \citet{Epstein2011} guarantees that an unbounded symmetric operator~$\genFull$ of the kind defined in \sref{sec:WF} of this paper has countably many eigenvalues $\{-\eVal_0,-\eVal_1,-\eVal_2,\ldots\}$, which are real and non-positive, satisfying
\begin{equation}
	0 \leq \eVal_0 \leq \eVal_1 \leq \eVal_2 \leq \cdots,
\end{equation}
with $\eVal_n\to\infty$ as $n\to\infty$.
An eigenfunction $\eFun_n \colon \simplex_{\numAlleles-1} \to \R$ with eigenvalue $-\eVal_n$ satisfies 
\begin{equation}\label{eq_eigenvalue_equation}
	\genFull \eFun_n (\xVec) = - \eVal_n \eFun_n (\xVec),
\end{equation}
and, furthermore, $\eFun_n (\xVec)$ is an element of the Hilbert space $L^2\big(\simplex_{\numAlleles-1}, \pi(\xVec) \big)$ of functions square integrable with respect to the density $\pi(\xVec)$, equipped with the canonical inner product~$\langle \cdot , \cdot \rangle_\pi$.
If $\genFull$ is symmetric with respect to $\pi(\xVec)$, then its eigenfunctions are orthogonal with respect to $\pi(\xVec)$:
\begin{equation}
	 \langle B_n, B_m \rangle_\pi := \int_{\simplex_{\numAlleles-1}} \eFun_n(\xVec) \eFun_m(\xVec) \pi(\xVec) d\xVec = \delta_{n,m} d_n,
\end{equation}
where $\delta_{n,m}$ is the Kronecker delta and $d_n$ are some constants. In the cases considered in this paper, the eigenfunctions form a basis of the Hilbert space $L^2\big(\simplex_{\numAlleles-1}, \pi(\xVec) \big)$.

It follows from equation~\eqref{eq_eigenvalue_equation} that $\exp(- \eVal_n t) \eFun_n(\xVec)$ is a solution to the Kolmogorov backward equation~\eqref{eq_backward_equation}. By linearity of~\eqref{eq_backward_equation}, the sum of two solutions is again a solution. Combining the initial condition $p(0;\xVec,\yVec) = \delta(\xVec - \yVec)$ with the fact that $\{\eFun_\eIdx(\xVec)\}$ form a basis yields the following spectral representation of the transition density:
\begin{equation}\label{eq_tdf}
	p(t;\xVec,\yVec) = \sum_{n=0}^\infty \frac{1}{d_n} e^{- \eVal_n t} \eFun_n (\xVec) \eFun_n (\yVec) \pi(\yVec).
\end{equation}
The initial condition being the Dirac delta $\delta(\xVec - \yVec)$ corresponds to the frequency at time zero being $\xVec$.

\subsection{Univariate Jacobi polynomials}

The univariate Jacobi polynomials play an important role throughout this paper.  Here we review some key facts about this particular type of classical orthogonal polynomials. An excellent treatise on univariate orthogonal polynomials can be found in \citet{Szego1939} and a comprehensive collection of useful formulas can be found in \citet[Chapter~22]{Abramowitz1965}.

The Jacobi polynomials $p^{(a,b)}_{n}(z)$, for $z \in [-1,1]$,  satisfy the differential equation
\begin{equation}
(1-z^2) \frac{d^2f(z)}{dz^2} + [b-a-(a+b+2)z]\frac{d f(z)}{dz} + n(n+a+b+1) f(z) = 0.
\label{eq_diffeq_jp}
\end{equation}
For given $a,b > -1$, the set $\{p^{(a,b)}_{n}(z)\}_{n=0}^\infty$ forms an orthogonal system on the interval $[-1,1]$ with respect to the weight function $(1-z)^a (1+z)^b$. For a more convenient correspondence with the diffusion parameters, we define the following modified Jacobi polynomials, for $x\in [0,1]$ and $a,b > 0$:
\begin{equation}
\uJacobi^{(a,b)}_{n}(x) = p^{(b-1,a-1)}_{n}(2x-1).
\label{eq_R}
\end{equation}
This definition is slightly different from that adopted by \citet{Griffiths2010}.

Equation \eqref{eq_diffeq_jp} implies that the modified Jacobi polynomials $\uJacobi^{(a,b)}_{n}(x)$, for $x\in [0,1]$, satisfy the differential equation
\begin{equation}\label{eq_R_ode}
	x(1-x) \frac{d^2f(x)}{dx^2} + [a-(a+b)x]\frac{d f(x)}{dx} + n(n+a+b-1) f(x) = 0.
\end{equation}
For fixed $a,b > 0$, the set $\{\uJacobi^{(a,b)}_{n}(x)\}_{n=0}^\infty$ forms an orthogonal system on $[0,1]$ with respect to the weight function $x^{a-1} (1-x)^{b-1}$.  More precisely,
\begin{equation}\label{eq_R_orthogonality}
	\int_{0}^1 \uJacobi^{(a,b)}_{n}(x) \uJacobi^{(a,b)}_{m}(x) x^{a-1}(1-x)^{b-1} dx = \delta_{n,m} \uJacLength_n^{(a,b)},
\end{equation}
where $\delta_{n,m}$ denotes the Kronecker delta and 
\begin{equation}
\uJacLength_n^{(a,b)}=\frac{\Gamma(n+a)\Gamma(n+b)}{(2n+a+b-1)\Gamma(n+a+b-1) \Gamma(n+1)}.
\label{eq_cnab}
\end{equation}
Note that $\{\uJacobi^{(a,b)}_{n}(x)\}_{n=0}^\infty$ form a complete basis of the Hilbert space $L^2([0,1],x^{a-1}(1-x)^{b-1})$.

For $n\geq 1$, the modified Jacobi polynomial $\uJacobi^{(a,b)}_{n}(x)$ satisfies the recurrence relation
\begin{equation}\label{eq:R_recurrence}
	\begin{split}
		x \uJacobi^{(a,b)}_{n}(x) =\, & \frac{(n + a - 1)(n + b - 1) }{(2 n + a + b - 1) (2 n + a + b - 2)} \uJacobi^{(a,b)}_{n-1}(x)\\
			& + \left[\frac{1}{2} - \frac{b^2 - a^2 - 2 (b - a)}{2 (2 n + a + b) (2 n + a + b - 2)} \right] \uJacobi^{(a,b)}_{n}(x)\\
			& + \frac{(n+1) (n + a + b - 1)}{(2 n + a + b) (2 n + a + b - 1)} \uJacobi^{(a,b)}_{n+1}(x),
	\end{split}
\end{equation}
while, for $n=0$,
\begin{equation}\label{eq:R_recurrence_initial}
	x \uJacobi^{(a,b)}_{0}(x) =  \frac{a}{a+b} \uJacobi^{(a,b)}_{0}(x) + \frac{1}{a+b} \uJacobi^{(a,b)}_{1}(x).
\end{equation}
Also, note that $\uJacobi^{(a,b)}_{0}(x) \equiv 1$. These recurrence relations play an important role in the work of \citet{Song2012}, and the multivariate analogues, discussed later in \sref{sec:recurrence}, are similarly important for the present work.

The modified Jacobi polynomials satisfy other interesting relations, one of them being the following:
\begin{equation}\label{eq_increase_beta}
	\begin{split}
		\uJacobi^{(a,b)}_n(x) = \frac{n + a + b -1}{2n + a + b-1} \uJacobi^{(a,b+1)}_n(x) - \1_{\{n>0\}}\frac{n + a - 1}{2n + a + b-1} \uJacobi^{(a,b+1)}_{n-1}(x).\\
	\end{split}
\end{equation}
Using this identity, polynomials with parameter $b$ can be related to polynomials with parameter $b+1$.  We utilize this relation later.

\subsection{A review of the $\numAlleles=2$ case}
\label{sec:review_K2}

To motivate the  approach to be employed in the general case, we briefly review the work of \citet{Song2012} for deriving the transition density function in the diallelic  ($\numAlleles=2$) case.
The vector of mutation rates is given by $\mutVec = (\alpha,\beta)$, while the symmetric matrix describing the general diploid selection scheme can be parametrized as
\begin{equation}
	\selMatrix=\left(
	\begin{array}{cc}
		2 \sigma & 2 \sigma h \\
		2 \sigma h & 0
	\end{array}\right),
\end{equation}
where $\sigma$ is the selection strength and $h$ the dominance parameter.  
For $\numAlleles=2$, the diffusion is one dimensional and the simplex $\simplex_1$ is equal to the unit interval $[0,1]$.
With $x$ denoting $\xCmp_1$, the generator~\eqref{eq_backward_generator} reduces to
\begin{equation}
	\genFull f(x) = \frac{1}{2}x(1-x) \frac{\partial^2}{\partial x^2} f(x) + \Big\{ \frac{1}{2}[\alpha - (\alpha + \beta) x] +2\sigma x(1-x)[x + h (1-2x)] \Big\} \frac{\partial}{\partial x} f(x).
\end{equation}

In the neutral case (i.e., $\sigma = 0$), the modified Jacobi polynomials $\uJacobi^{(\alpha,\beta)}_n(x)$ are eigenfunctions of the diffusion generator with eigenvalues $\lambda_n^{(\alpha,\beta)} = \frac{1}{2}n(n - 1 +\alpha + \beta)$.  Hence, a spectral representation of the transition density function 
can be readily obtained via \eqref{eq_tdf}. 

In the non-neutral case (i.e., $\sigma \neq 0$), consider the functions $\modJacobi_n(x) = e^{-\meanFitness(x)/2}\uJacobi^{(\alpha,\beta)}_n(x)$, which form an orthogonal basis of the Hilbert space $L^2([0,1],e^{\meanFitness(x)}x^{\alpha-1}(1-x)^{\beta-1})$, where $e^{\meanFitness(x)}x^{\alpha-1}(1-x)^{\beta-1}$ corresponds to the stationary distribution of the non-neutral diffusion, up to a multiplicative constant. Since the eigenfunctions $\eFun_n(x)$ of the diffusion generator are elements of this Hilbert space, we can pose an expansion $\eFun_n(x) = \sum_{m=0}^\infty w_{n,m} \modJacobi_m(x)$ in terms of the basis functions $\modJacobi_n(x)$, where $w_{n,m}$ are to be determined.  Then, the eigenvalue equation $\genFull \eFun_n(x) = - \eVal_n \eFun_n(x)$ implies the algebraic equation
\begin{equation}
\sum_{m=0}^\infty w_{n,m} \left[ \lambda_m^{(\alpha,\beta)}  + Q(x;\alpha,\beta,\sigma,h)\right] \uJacobi^{(\alpha,\beta)}_m(x) = \eVal_n \sum_{m=0}^\infty w_{n,m} \uJacobi^{(\alpha,\beta)}_m(x),
\end{equation}
where $Q(x;\alpha,\beta,\sigma,h)$  is a polynomial in $x$ of degree four.  Utilizing the recurrence relations in \eqref{eq:R_recurrence} and \eqref{eq:R_recurrence_initial}, one can then arrive at a linear system $M \bfw_n = \Lambda_n \bfw_n$, where $\bfw_n=(w_{n,0}, w_{n,1}, w_{n,2},\ldots)$ is an infinite-dimensional vector of variables and $M$ is a sparse infinite-dimensional matrix with entries that depend on the index $n$ and the parameters $\alpha,\beta,\sigma,h$ of the model.  The infinite linear system $M \bfw_n = \Lambda_n \bfw_n$ is approximated by a finite-dimensional truncated linear system
\[
M^\supN \bfw^\supN_n = \Lambda^\supN_n \bfw^\supN_n,
\]
where $\bfw_n^{(\cutoff)}=(w_{n,0}^\supN,w_{n,1}^\supN,\ldots,w_{n,\cutoff-1}^\supN)$ and $M^\supN$ is the submatrix of $M$ consisting of its first $\cutoff$ rows and $\cutoff$ columns.  This finite-dimensional linear system can be easily solved using standard linear algebra to obtain the eigenvalues $\eVal_n^\supN$ and the eigenvectors $\bfw_n^\supN$ of $M^\supN$.
\citeauthor{Song2012} observed that $\Lambda_n^\supN$ and $w_{n,m}^\supN $ converge very rapidly as the truncation level  $\cutoff$ increases.
Finally, the coefficients $w_{n,m}^\supN$ can be used to approximate the eigenfunctions $\eFun_n(x)$, and, together with the eigenvalues $\eVal_n^\supN$, an efficient approximation of the transition density function can be obtained via \eqref{eq_tdf}. 

\section{The Neutral Case with an Arbitrary Number of Alleles}
\label{sec_neutral_case}

In this section, we describe the spectral representation of the transition density of a neutral PIM model with an arbitrary number $\numAlleles$ of alleles. As in the case of $\numAlleles=2$, reviewed in \sref{sec:review_K2}, for an arbitrary $\numAlleles$ the eigenfunctions in the neutral case can be used to construct the eigenfunctions in the case with selection.   The latter case is considered in \sref{sec:selected_case}.

\subsection{Multivariate Jacobi polynomials}
\label{subsec_multiJacob}

In what follows, let $\N_0 = \{0,1,2,\ldots\}$ denote the set of non-negative integers.
As in  \citet{Griffiths2011}, we define the following system of multivariate orthogonal polynomials in $\numAlleles-1$ variables:
\begin{definition}\label{def_multi_jacobi}
For each vector $\nIdx=(\nIdxCmp_1,\ldots,\nIdxCmp_{\numAlleles-1}) \in \N_0^{\numAlleles-1}$ and $\mutVec=(\mutCmp_1,\ldots,\mutCmp_\numAlleles)\in \R_{\geq 0}^\numAlleles$, the orthogonal polynomial $\jacobi^\mutVec_\nIdx(\xVec)$ is defined as
\begin{equation}\label{eq_def_polynomials}
	\jacobi^\mutVec_\nIdx(\xVec) =  \prod_{j=1}^{\numAlleles-1} \left[ \bigg(1-\frac{\xCmp_j}{1-\sum_{i=1}^{j-1}\xCmp_i}\bigg)^{\nTot_j} \uJacobi_{\nIdxCmp_j}^{(\mutCmp_j,\mutTot_j + 2\nTot_j)}\bigg(\frac{\xCmp_j}{1-\sum_{i=1}^{j-1}\xCmp_i}\bigg)  \right],
\end{equation}
where $\nTot_j = \sum^{\numAlleles-1}_{i=j+1} \nIdxCmp_i$ and $\mutTot_j = \sum^{\numAlleles}_{i=j+1} \mutCmp_i$.
\end{definition}

For $\xVec=(\xCmp_1,\ldots,\xCmp_{\numAlleles-1})\in \simplex_{\numAlleles-1}$, let $\statNeutr(\xVec)$ denote an unnormalized density of the  Dirichlet distribution with parameter~$\mutVec=(\mutCmp_1,\ldots,\mutCmp_\numAlleles)$:
\begin{equation}
\statNeutr(\xVec) = \prod_{i=1}^K \xCmp_i^{\mutCmp_i-1},
\label{eq:Pi0}
\end{equation}
where $\xCmp_K=1-|\xVec|$.
The following lemma, a proof of which is provided in Appendix~\ref{app_proofs}, states that the above multivariate Jacobi polynomials $\jacobi^\mutVec_\nIdx(\xVec)$ are orthogonal with respect to $\statNeutr(\xVec)$:

\begin{lemma}\label{lem_orthogonality}
For $\nIdx,\mIdx \in \N_0^{\numAlleles-1}$,
\begin{equation}
	\begin{split}
		\int_{\simplex_{\numAlleles-1}}  \jacobi^\mutVec_\nIdx (\xVec) \jacobi^\mutVec_\mIdx (\xVec) \statNeutr(\xVec) d\xVec & = \delta_{\nIdx,\mIdx} \jacLength^\mutVec_\nIdx,
	\end{split}
\end{equation}
where $\delta_{\nIdx,\mIdx}=\prod_{i=1}^{\numAlleles-1} \delta_{\nIdxCmp_i,\mIdxCmp_i}$ and
\begin{equation}
	\jacLength^\mutVec_\nIdx := \prod_{i=1}^{\numAlleles-1} \uJacLength_{\nIdxCmp_i}^{(\mutCmp_i,\mutTot_i + 2\nTot_i)},
\label{eq_const}
\end{equation}
with $\uJacLength_{n}^{(a,b)}$ defined in \eqref{eq_cnab}.
\end{lemma}

\medskip\noindent\emph{Remark:}
	The multivariate Jacobi polynomials form a complete basis of $L^2\big(\simplex_{\numAlleles-1}, \statNeutr(\xVec)\big)$, the Hilbert space of functions on $\simplex_{\numAlleles-1}$ square integrable with respect to the unnormalized Dirichlet density $\statNeutr(\xVec)$.

\subsection{Recurrence relation for multivariate Jacobi polynomials}
\label{sec:recurrence}
Recall that the univariate Jacobi polynomials satisfy the recurrence relation \eqref{eq:R_recurrence}.
Theorem~3.2.1 of \citet{Dunkl2001} guarantees that the multivariate Jacobi polynomials satisfy a similar recurrence relation. 
More precisely, we have the following lemma, the proof of which is provided in Appendix~\ref{app_proofs}:
\begin{lemma}
\label{lem_mult_recurrence}
Given $\nIdx=(\nIdxCmp_1,\ldots,\nIdxCmp_{\numAlleles-1})\in\N_0^{\numAlleles-1}$ and $\mIdx=(\mIdxCmp_1,\ldots,\mIdxCmp_{\numAlleles-1})\in\N_0^{\numAlleles-1}$, define $\nTot_j = \sum^{\numAlleles-1}_{i=j+1} \nIdxCmp_i$ and  $\mTot_j = \sum^{\numAlleles-1}_{i=j+1} \mIdxCmp_i$.
For given $i\in \{1,\ldots,\numAlleles-1\}$ and $\nIdx$,  $\jacobi^\mutVec_\nIdx (\xVec)$ satisfies the recurrence relation 
\begin{equation}
	\xCmp_i \jacobi^\mutVec_\nIdx (\xVec) = \sum_{\mIdx \in \M_i(\nIdx)} r^{(\mutVec,i)}_{\nIdx,\mIdx} \jacobi^\mutVec_\mIdx (\xVec),
\label{eq_multiJacobi_recurrence}
\end{equation}
where $r^{(\mutVec,i)}_{\nIdx,\mIdx}$ are known constants (provided in Appendix~\ref{app_proofs}) and 
\begin{equation}
	\M_i (\nIdx) := \Big\{ \mIdx \in \N_0^{\numAlleles-1} : \mTot_j = \nTot_j \text{ for all } j>i \text{ and } | \mTot_j - \nTot_j| \leq 1 \text{ for all } j \leq i\Big\}.
\label{eq_index_set}
\end{equation}
\end{lemma}

Impose an ordering on the $(\numAlleles-1)$-dimensional index vectors $\bfn\in \N_0^{\numAlleles-1}$.  Then, the recurrence~\eqref{eq_multiJacobi_recurrence} can be represented as
\begin{equation}
  	\xCmp_i \jacobi^\mutVec_\nIdx (\xVec) = \sum_{\mIdx \in \M_i(\nIdx)}  [\polyMatrix^\mutVec_i]_{\nIdx,\mIdx} \jacobi^\mutVec_\mIdx (\xVec),
\end{equation}
where $\polyMatrix^\mutVec_i$ corresponds to an infinite dimensional matrix in which columns and rows are indexed by the ordered $(\numAlleles-1)$-tuples, and the \mbox{$(\nIdx,\mIdx)$-th} entry is defined as
\begin{equation}
[\polyMatrix^\mutVec_i]_{\nIdx,\mIdx}=\begin{cases}
r^{(\mutVec,i)}_{\nIdx,\mIdx}, & \text{if $\mIdx\in\M_i(\nIdx)$},\\
0, & \text{otherwise}.
\end{cases}  
\label{eq:Gi}
\end{equation}
Note that for each given $\nIdx$, the number of non-zero entries in every row of $\polyMatrix^\mutVec_i$ is finite.
One can deduce the following corollary from the new representation:
\begin{corollary}
\label{cor_polyMatrix}
Let $\bfa=(a_\nIdx)_{\nIdx\in\N_0^{\numAlleles-1}}$ be such that $\sum_{\nIdx\in\N_0^{\numAlleles-1}} a_\nIdx^2 \jacLength^\mutVec_\nIdx < \infty$.
Then
\begin{equation}
\xCmp_i \cdot \sum_{\nIdx\in\N_0^{\numAlleles-1}} a_\nIdx \jacobi_\nIdx^{\mutVec}(\xVec) = \sum_{\nIdx\in\N_0^{\numAlleles-1}} b_\nIdx \jacobi_\nIdx^{\mutVec}(\xVec), 
\end{equation}
where  $(b_\nIdx)_{\nIdx\in\N_0^{\numAlleles-1}}=\bfa\cdot\polyMatrix^\mutVec_i$.
\end{corollary}

\medskip\noindent\emph{Remark:}	Since under multiplication $x_i$ commutes with $x_j$ for $1 \leq i,j \leq \numAlleles-1$, the corresponding matrices $\polyMatrix^\mutVec_i$ and $\polyMatrix^\mutVec_j$ also commute.

\subsection{Eigenfunctions of the neutral generator $\genNeutral$} 
It is well known that the stationary distribution of the Wright-Fisher diffusion under a \textit{neutral} PIM model is the Dirichlet distribution \citep{Wright1949}.
The density of the Dirichlet distribution is a weight function with respect to which the associated diffusion generator $\genNeutral$ is symmetric.
As discussed in \sref{subsec_multiJacob}, the multivariate Jacobi polynomials $\jacobi_\nIdx^{\mutVec}(\xVec)$ are orthogonal with respect to the weight function $\statNeutr(\xVec)$, which is equal to the density of the Dirichlet distribution up to a multiplicative constant.  Given the discussion in \sref{sec:background_spec_rep}, one might then suspect that $\jacobi_\nIdx^{\mutVec}(\xVec)$ are potential eigenfunctions of $\genNeutral$. The following lemma establishes that this is indeed the case:

\begin{lemma}
\label{lem_eigen_neutral}
For all $\nIdx\in\N_0^{\numAlleles-1}$, the multivariate Jacobi polynomials $\jacobi_\nIdx^{\mutVec}(\xVec)$ satisfy
\begin{equation}
	\genNeutral \jacobi_\nIdx^{\mutVec}(\xVec) = -\eVecNeutr^\mutVec_{\vert \nIdx \vert} \jacobi_{\nIdx}^{\mutVec}(\xVec),
\end{equation}
where
\begin{equation}
	\eVecNeutr^\mutVec_{\vert \nIdx \vert}=\frac{1}{2}|\nIdx|(|\nIdx|-1+|\mutVec|).
	\label{eq:neutral_eVal}
\end{equation}
That is, $\jacobi_\nIdx^{\mutVec}(\xVec)$ are eigenfunctions of $\genNeutral$ with eigenvalues  $-\eVecNeutr^\mutVec_{\vert \nIdx \vert}$.
\end{lemma}
A proof of this lemma is deferred to Appendix~\ref{app_proofs}.  We conclude this section with a few comments.

\medskip\noindent\emph{Remarks:}
\begin{enumerate}[i)]
\item Substituting the eigenvalues and eigenfunctions into the spectral representation~\eqref{eq_tdf}, we obtain 
\begin{equation}\label{eq_tdf_PIM}
p(t;\xVec,\yVec) = \sum_{\nIdx\in\N_0^{\numAlleles-1}} \frac{1}{\jacLength_\nIdx^\mutVec}\, e^{- \eVecNeutr^\mutVec_{\vert \nIdx \vert} t} \jacobi_\nIdx^{\mutVec} (\xVec) \jacobi_\nIdx^{\mutVec} (\yVec) \statNeutr(\yVec).
\end{equation}
\item 
For every $\nIdx\in \N_0^{\numAlleles-1}$, note that $\eVecNeutr^\mutVec_{\vert \nIdx \vert}$ only depends on the norm $|\nIdx|$, which implies degeneracy in the spectrum of $\genNeutral$.
\cite{Griffiths1979} constructed orthogonal kernel polynomials indexed by $|\nIdx|$, that is the sum over all orthogonal polynomials with index summing to $|\nIdx|$, and obtained the transition density expansion~\eqref{eq_tdf_PIM}.
\end{enumerate}

\section{A General Diploid Selection Case with an Arbitrary Number of Alleles}
\label{sec:selected_case}
In this section, we derive the spectral representation of the transition density function of the Wright-Fisher diffusion under a $\numAlleles$-allelic PIM model with general diploid selection.  This work extends the work of \cite{Song2012}, the special case of $\numAlleles=2$ briefly summarized in \sref{sec:review_K2}, to an arbitrary number $\numAlleles$ of alleles.  The recurrence relation presented in Lemma~\ref{lem_mult_recurrence} plays a crucial role in the following derivation.

Recall that the backward generator for the full model is  $\genFull =\genNeutral+\genSel$, where
$\genNeutral$ corresponds to the generator under neutrality and $\genSel$ corresponds to the contribution from selection. The diffusion has a unique stationary density [see  \citet{Ethier1994} or~\citet{Barbour2000}] proportional to
\begin{equation}\label{eq_stat_dens_sel}
	\stat(\xVec):= e^{\meanFitness(\xVec)} \statNeutr(\xVec),
\end{equation}
where $\statNeutr(\xVec)$ is defined in \eqref{eq:Pi0} and  $\meanFitness(\xVec)$ is the mean fitness defined in~\eqref{eq:meanFitness}. As mentioned in \sref{sec:background_spec_rep}, $\genFull $ is symmetric with respect to $\stat(\xVec)$.  For $\eIdx\in\N_0$, we aim to find the eigenvalues $-\eVal_{\eIdx}$ and the eigenfunctions $\eFun_{\eIdx}$ of $\genFull $ such that 
\begin{equation}\label{eq:eigen_L}
	\genFull  \eFun_{\eIdx}(\xVec)=-\eVal_{\eIdx} \eFun_{\eIdx}(\xVec).
\end{equation}
 By convention, we place $\eVal_\eIdx$ in non-decreasing order. The symmetry of $\genFull $ implies that $\{\eFun_{\eIdx}(\xVec)\}$ form an orthogonal system with respect to $\stat(\xVec)$, that is
\begin{equation}
	\int_{\simplex_{\numAlleles-1}}\eFun_{\eIdx}(\xVec)\eFun_{\eIdxM}(\xVec)\stat(\xVec)d\xVec \propto \delta_{\eIdx,\eIdxM}.
\end{equation}
Such a system of orthogonal functions, however, is not unique. The orthogonality of $\{\jacobi_\nIdx^{\mutVec}\}$ with respect to $\statNeutr$, established in Lemma~\ref{lem_orthogonality}, can be used to show that the functions 
\begin{equation}
\modJacobi_\nIdx(\xVec) := \jacobi_{\nIdx}^{\mutVec}(\xVec)e^{-\meanFitness(\bf  \xVec)/2}
\label{eq:modJacobi}
\end{equation}
are orthogonal with respect to $\stat$, as are $\eFun_\eIdx(\xVec)$. Furthermore, the fact that $\{\jacobi_\nIdx^{\mutVec}(\xVec)\}$ form a complete basis of $L^2(\simplex_{\numAlleles-1},\statNeutr(\xVec))$ means that $\{\modJacobi_\nIdx(\xVec)\}$ is a complete basis of $L^2\big(\simplex_{\numAlleles-1},\stat(\bf  x)\big)$. 
Since $B_n \in L^2\big(\simplex_{\numAlleles-1},\stat(\bf  x)\big)$, we thus seek to represent $\eFun_\eIdx(\xVec)$ as linear combination of the basis $\modJacobi_\nIdx(\xVec)$:
\begin{equation}\label{eq_linear_rep}
	\eFun_{\eIdx}(\xVec)=\sum_{\mIdx\in\N_0^{\numAlleles-1}}\uCoeff_{\eIdx,\mIdx}\modJacobi_{\mIdx}(\xVec),
\end{equation}
where $\uCoeff_{\eIdx,\mIdx}$ are some constants to be determined.

Define an index set $\iSet=\cup_{L=0}^{4}\left\{ 1,\dots,\numAlleles-1\right\} ^{L}$ and for $\iIdx=(\iCmp_{1},\dots,\iCmp_{L})\in \iSet$, define $\xVec_{\iIdx}=\xCmp_{\iCmp_{1}}\cdots \xCmp_{\iCmp_{L}}$.
 We have the following theorem for solving the eigensystem associated with the full generator $\genFull$:

\begin{theorem}
\label{thm_main}
For all $n\in\N_0$, the eigenfunction $\eFun_{\eIdx}(\xVec)$ of $\genFull $ can be represented by~\eqref{eq_linear_rep}. The corresponding eigenvalues $-\eVal_\eIdx$ and the coefficients $\uCoeff_{\eIdx,\mIdx}$ can be found by solving the infinite-dimensional eigensystem
\begin{equation}\label{eq_matrix_eqn}
	\eVec_\eIdx{\bf M} = \eVec_\eIdx\eVal_\eIdx,
\end{equation}
where $\eVec_\eIdx=(\uCoeff_{\eIdx,\mIdx})_{\mIdx\in\N_0^{\numAlleles-1}}$ and
\begin{equation}
{\bf M} = \mathrm{diag} \left(\{\lambda^\mutVec_{|\mIdx|}\}_{\mIdx\in\N_0^{\numAlleles-1}}\right) + \sum_{\iIdx\in\iSet} q(\iIdx) \polyMatrix^\mutVec_\iIdx.
\end{equation} 
Here, $\lambda^\mutVec_{|\mIdx|}$ is defined as in \eqref{eq:neutral_eVal} and, for $\iIdx=(\iCmp_1,\ldots,\iCmp_L)\in\iSet$, we define
$\polyMatrix^\mutVec_\iIdx=\polyMatrix^\mutVec_{\iCmp_1}\cdots\polyMatrix^\mutVec_{\iCmp_L}$, where $\polyMatrix^\mutVec_{\iCmp}$ is given in \eqref{eq:Gi}.
When $L=0$, $\polyMatrix^\mutVec_\iIdx$ is defined to be the identity matrix.  Explicit expressions of the constants $q(\iIdx)$ are provided in Appendix~\ref{app_poly_coeff}. 
\end{theorem}

\medskip\noindent\emph{Remarks:}
\begin{enumerate}[i)]
\item Although $\bf M$ is infinite dimensional, it is in fact sparse with only finitely many non-zero entries in every row and column.
\item Equation~\eqref{eq_matrix_eqn} implies that $\eVal_\eIdx$ and $\eVec_\eIdx$ are in fact the left eigenvalues and eigenvectors of ${\bf M}$. To solve the eigensystem in practice requires some truncation of the matrix to finite dimensions, as in the $\numAlleles=2$ case described in \sref{sec:review_K2}.   
For a given $n\in\N_0$, we would like both $\eVal_\eIdx$ and $\uCoeff_{\eIdx,\mIdx}$ to converge as the truncation level increases.  In \sref{sec_simulations}, we demonstrate that this is indeed the case  using empirical examples.
\end{enumerate}
We now provide a proof of the theorem.
\begin{proof}[Proof of Theorem~\ref{thm_main}]
Substituting \eqref{eq_linear_rep} into \eqref{eq:eigen_L} we obtain
\[
\sum_{\kIdx \in\N_0^{\numAlleles-1}} \uCoeff_{\eIdx,\kIdx} \genFull \modJacobi_{\kIdx}(\xVec) = - 
\sum_{\kIdx \in\N_0^{\numAlleles-1}} \eVal_{\eIdx} \uCoeff_{\eIdx,\kIdx} \modJacobi_{\kIdx}(\xVec).
\]
It is shown in Appendix~\ref{app_derivation_polynomial} that 
\begin{equation}
\genFull \modJacobi_\kIdx(\xVec) = - e^{-\meanFitness(\xVec)/2} \left[ \lambda^\mutVec_{\vert \kIdx \vert}\jacobi_{\kIdx}^{\mutVec}(\xVec) + Q(\xVec;\selMatrix,\mutVec) \jacobi_{\kIdx}^{\mutVec}(\xVec) \right],
\label{eq_L_on_H}
\end{equation}
where
\[
Q(\xVec;\selMatrix,\mutVec) = \frac{1}{2} \left[ \sum_{i=1}^{\numAlleles} \xCmp_{i} \selCmp_{i}^{2}(\xVec) + \sum_{i=1}^{\numAlleles} \mutCmp_{i} \selCmp_{i}(\xVec) + \sum_{i=1}^{\numAlleles} \xCmp_{i} \selCmp_{i,i} - (1+|\mutVec|) \meanFitness(\xVec) - \meanFitness(\xVec)^{2} \right],
\]
with $\selCmp_{i}(\xVec)$ and $\meanFitness(\xVec)$ defined as in \eqref{eq:marginalFitness} and \eqref{eq:meanFitness}, respectively.  Thus, one arrives at the following equation:
\begin{equation}
	\sum_{\kIdx \in\N_0^{\numAlleles-1}} \eVal_\eIdx \uCoeff_{\eIdx,\kIdx} \jacobi_{\kIdx}^{\mutVec}(\xVec) = \sum_{\kIdx \in\N_0^{\numAlleles-1}} \uCoeff_{\eIdx,\kIdx} \big[\lambda^\mutVec_{\vert \kIdx \vert} \jacobi_{\kIdx}^{\mutVec}(\xVec) + Q(\xVec;\selMatrix,\mutVec) \jacobi_{\kIdx}^{\mutVec}(\xVec) \big].
\label{eq:equate_coeff}
\end{equation}
We solve the equation by first representing $Q(\xVec;\selMatrix,\mutVec) \jacobi_{\kIdx}^{\mutVec}(\xVec)$ as a finite linear combination of $\{\jacobi_{\bf k}^{\mutVec}(\xVec)\}_{{\bf k}\in\N_0^{\numAlleles-1}}$. Observe that $Q$ is in fact a fourth-order polynomial in $\xVec$. Collecting terms, $Q$ can be written in the form
\begin{equation}
Q(\xVec;\selMatrix,\mutVec) = \sum_{\iIdx\in\iSet}q(\iIdx)\xVec_{\iIdx},
\label{eq_Q}
\end{equation}
for the constants $q(\iIdx)$ given in Appendix~\ref{app_poly_coeff}. Applying Corollary~\ref{cor_polyMatrix} recursively, we obtain
\begin{equation}
Q(\xVec;\selMatrix,\mutVec) \jacobi_{\kIdx}^{\mutVec}(\xVec)  =  \sum_{\iIdx\in\iSet} q(\iIdx) \sum_{\lIdx\in\N_0^{\numAlleles-1}}[\polyMatrix^\mutVec_\iIdx]_{\kIdx,\lIdx}\jacobi_\lIdx^{\mutVec}(\xVec).
\end{equation}
Finally, substituting this equation into~\eqref{eq:equate_coeff}, multiplying both sides of \eqref{eq:equate_coeff} by $\jacobi_\mIdx^{\mutVec}(\xVec)$, and integrating with respect to $\statNeutr(\xVec)$ over the simplex $\simplex_{\numAlleles-1}$ yields the matrix equation~\eqref{eq_matrix_eqn}.
\end{proof}

\section{Empirical Results and Applications}
\label{sec_simulations}
In this section, we study the convergence behavior of the eigenvalues and eigenvectors as we approximate the solutions of~\eqref{eq_matrix_eqn}. Further, we show how the spectral representation can be employed to obtain the transient and stationary density explicitly (especially the normalizing constant), and to characterize the convergence rate of the diffusion to stationarity. A Mathematica implementation of the relevant formulas for computing the spectral representation is available from the authors upon request.

\subsection{Convergence of the eigenvalues and eigenvectors}
\label{subsec_cvg}
In what follows we order the Jacobi polynomials according to the \emph{graded lexicographic ordering} of their corresponding indices. Thus $\jacobi^{\mutVec}_{\nIdx_1} < \jacobi^{\mutVec}_{\nIdx_2}$ if
\begin{itemize}
\item $\vert \nIdx_1 \vert < \vert \nIdx_2 \vert$, or
\item $\vert \nIdx_1 \vert = \vert \nIdx_2 \vert$ and $\nIdx_1$ is lexicographically smaller than $\nIdx_2$.
\end{itemize}
Fix $\numAlleles$ and note that, for a given \emph{truncation level} $\matrixCutoff \in \N_0$ and $l \in \N_0$, there are ${l + \numAlleles -2 \choose \numAlleles -2}$ polynomials $\jacobi^{\mutVec}_\nIdx$ with $\vert \nIdx \vert = l$, and $\smallerIndices(\matrixCutoff) := {\matrixCutoff + \numAlleles -1 \choose \numAlleles -1}$ polynomials with index $\vert \nIdx \vert \leq \matrixCutoff$. For the computations in the rest of this section we chose $\numAlleles = 3$, unless otherwise stated.

Now, one can obtain a finite-dimensional linear system approximating~\eqref{eq_matrix_eqn} by truncation, that is, taking only those entries in ${\bf M}$ and $\eVec_n$ whose associated index vectors satisfy $\vert \nIdx \vert \leq D$. More explicitly, with ${\bf M}^{[\matrixCutoff]} = \big([{\bf M}]_{{\bf k},{\bf l}}\big)\in\R^{\smallerIndices(\matrixCutoff) \times \smallerIndices(\matrixCutoff)}$ and $\eVec_\eIdx^{[\matrixCutoff]}=(\uCoeff_{\eIdx,{\bf k}})\in\R^{\smallerIndices(\matrixCutoff)}$, where ${\bf k}, {\bf l} \in \N_0^{K-1}$ such that $\vert {\bf k} \vert, \vert {\bf l} \vert \leq \matrixCutoff$, the solutions of
\begin{equation}
\eVec_\eIdx^{[\matrixCutoff]} {\bf M}^{[\matrixCutoff]} = \eVec_\eIdx^{[\matrixCutoff]} \eVal_\eIdx^{[\matrixCutoff]}
\end{equation}
should approximate the solutions of the infinite system $\eVal_\eIdx$ and $\eVec_\eIdx$.
The convergence patterns of $\eVal_\eIdx^{[\matrixCutoff]}$ and $\uCoeff_{\eIdx,{\bf k}}^{[\matrixCutoff]}$ as $\matrixCutoff$ increases are exemplified in Figure~\ref{fig:convergence} for the parameters 
\begin{enumerate}[i)]
\item $K=3, \mutVec = (0.01,0.02,0.03)$, $\selMatrix = \selMatrix_1 :=\left(
\begin{array}{ccc}
	12 & 14 & 15 \\
	14 & 11 & 13\\
	15 & 13 & 0
\end{array}\right)$; and
\item $K=3, \mutVec = (0.01,0.02,0.03)$, $\selMatrix = \selMatrix_2 :=\left(
\begin{array}{ccc}
	120 & 140 & 150 \\
	140 & 110 & 130\\
	150 & 130 & 0
\end{array}\right)$.
\end{enumerate}
Figure~\ref{fig:convergence_more} displays the convergence behavior for the parameters
\begin{enumerate}[i)]
\setcounter{enumi}{2}
\item $K=3, \mutVec = (10,20,30)$, $\selMatrix = \selMatrix_1$; and
\item $K=4, \mutVec = (0.01,0.02,0.03,0.04)$, $\selMatrix_3=\left(
\begin{array}{cccc}
	12 & 14 & 15 & 16 \\
	14 & 11 & 10 & 13 \\
	15 & 10 & 9  & 14  \\
	16 & 13 & 14 & 0 
\end{array}\right)$.
\end{enumerate}

In all cases, $\eVal_\eIdx^{[\matrixCutoff]}$ and $\uCoeff_{\eIdx,{\bf k}}^{[\matrixCutoff]}$ converge with increasing truncation level to empirical limits. The eigenvalues $\eVal_\eIdx^{[\matrixCutoff]}$ decrease towards the empirical limit, whereas the coefficients $\uCoeff_{\eIdx,{\bf k}}^{[\matrixCutoff]}$ show oscillatory behavior before ultimately stabilizing. The rate of convergence is faster for smaller selection intensity. Varying the mutation parameters does not influence convergence behavior significantly. As expected, $\eVal_0^{[\matrixCutoff]}$ converges rapidly to zero in all cases, consistent with the fact that the diffusion has a stationary distribution. For a fixed $\eIdx$, $\eVal_\eIdx^{[\matrixCutoff]}$ and its associated coefficients $\uCoeff_{\eIdx,{\bf k}}^{[\matrixCutoff]}$ roughly converge at similar truncation levels.

\begin{figure}
\centerline{
	\subfigure[]{\includegraphics[width=.49\textwidth]{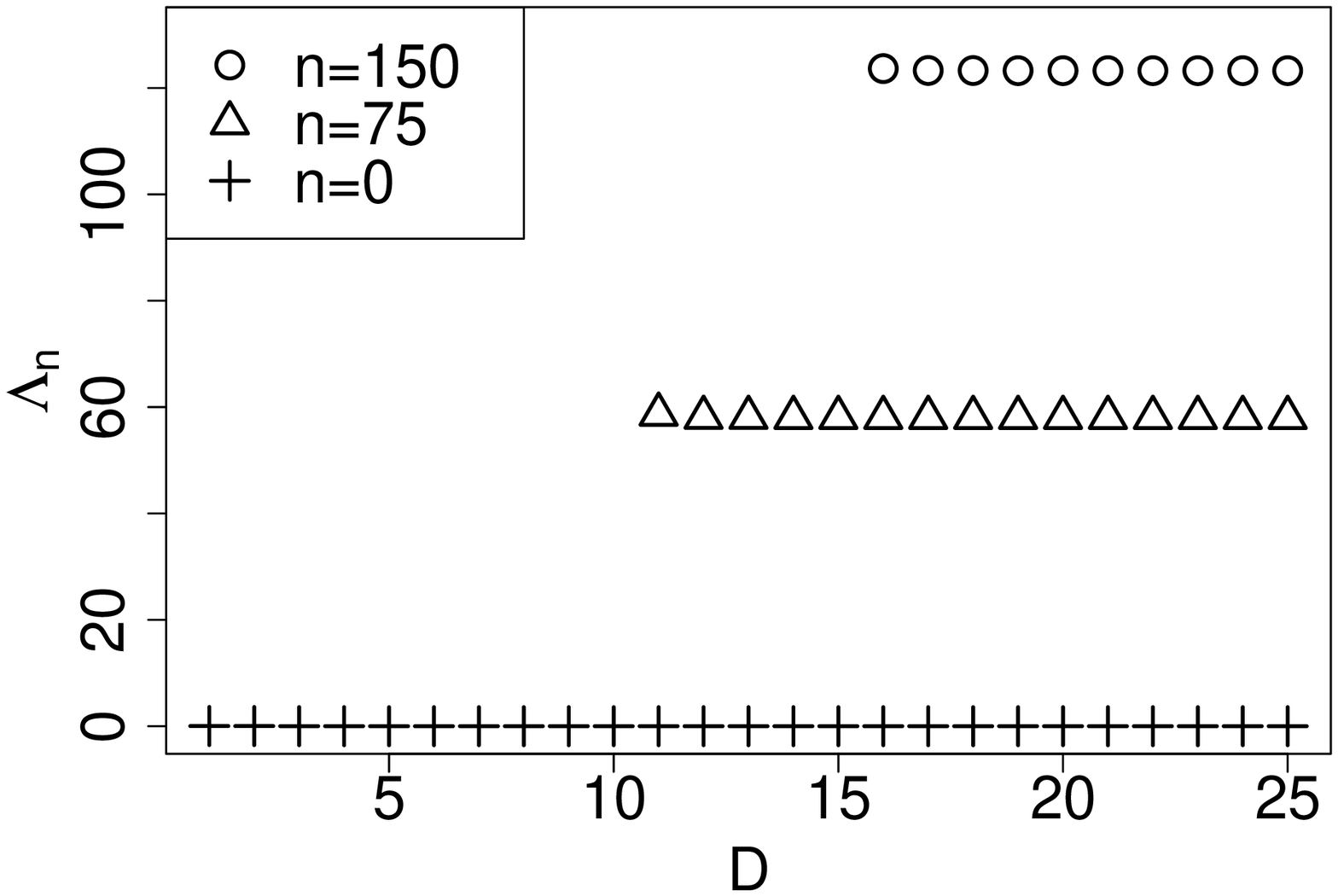}}
	\subfigure[]{\includegraphics[width=.49\textwidth]{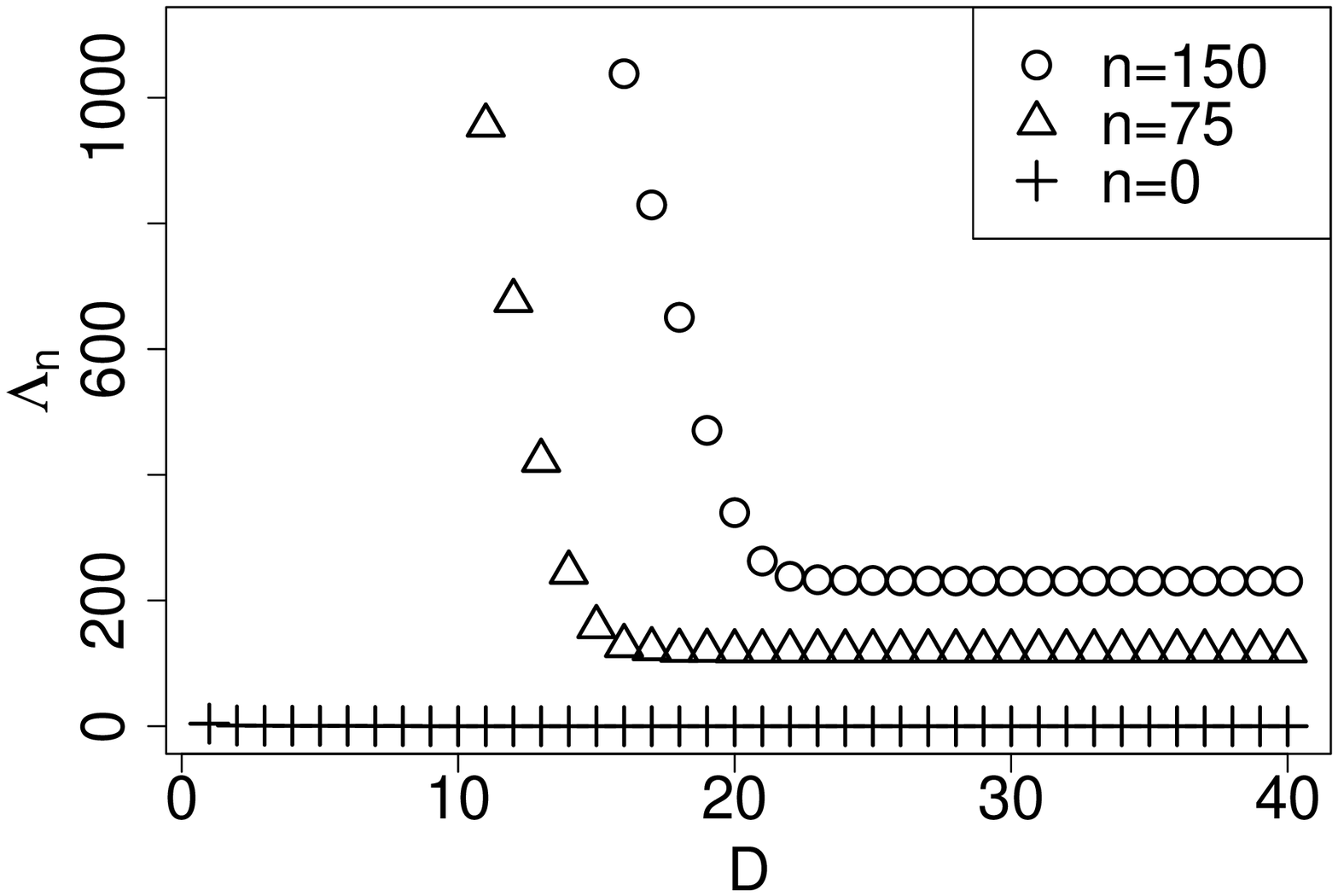}}
}
\centerline{
	\subfigure[]{\includegraphics[width=.49\textwidth]{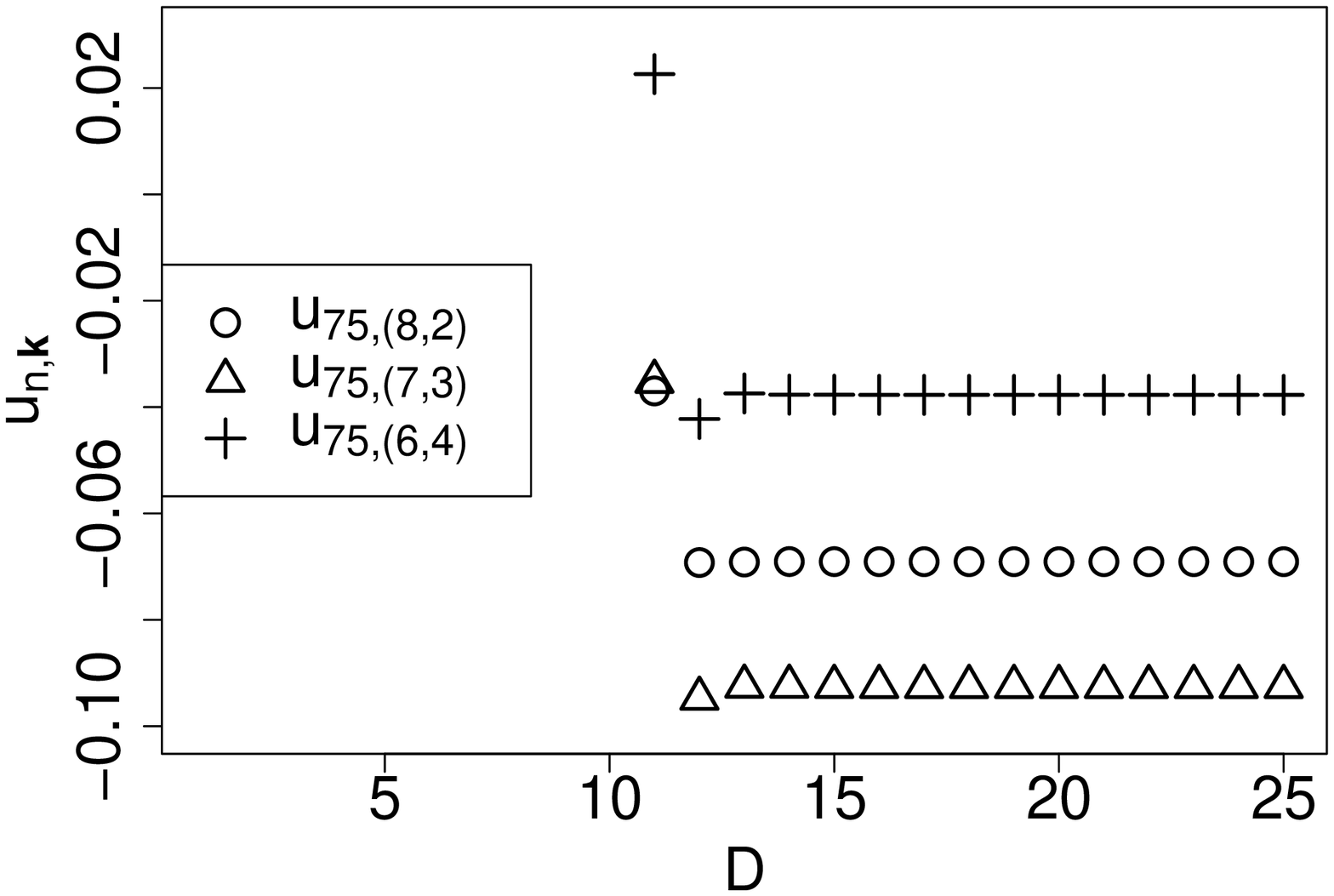}}
	\subfigure[]{\includegraphics[width=.49\textwidth]{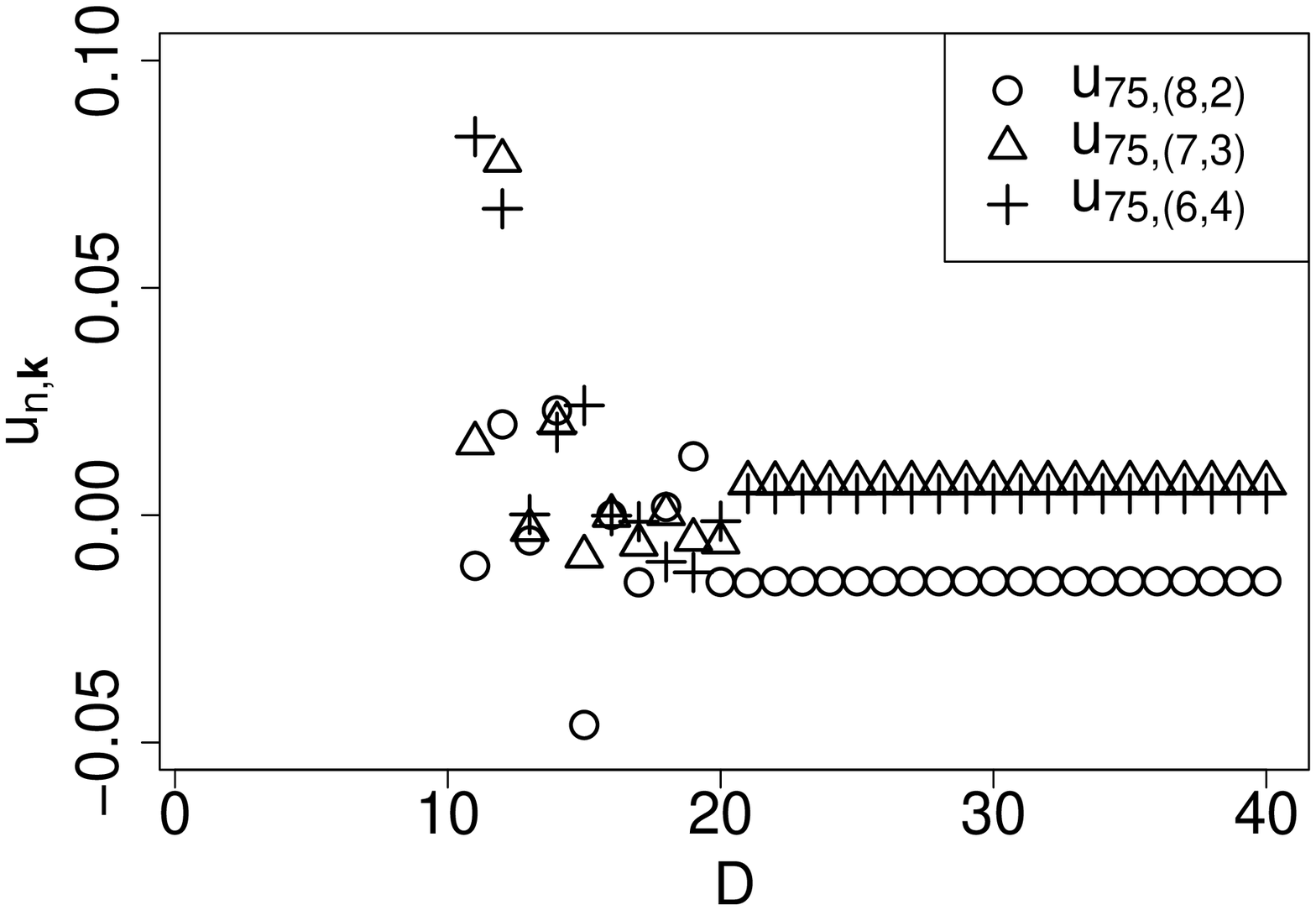}}
}
\caption{Convergence of the truncated eigenvalues $\eVal_{\eIdx}$ and coefficients of the eigenvectors $\eVec_{\eIdx}$ as the truncation level $D$ increases, for $\numAlleles=3$ with low mutation rates. Subfigures~(a) and~(b) show $\eVal_0^{[\matrixCutoff]}$, $\eVal_{75}^{[\matrixCutoff]}$, and $\eVal_{150}^{[\matrixCutoff]}$ for $\selMatrix=\selMatrix_1$ and $\selMatrix=\selMatrix_2$, respectively. Subfigures~(c) and~(d) show $\uCoeff_{75,(8,2)}^{[\matrixCutoff]}$, $\uCoeff_{75,(7,3)}^{[\matrixCutoff]}$, and $\uCoeff_{75,(6,4)}^{[\matrixCutoff]}$ for $\selMatrix=\selMatrix_1$ and $\selMatrix=\selMatrix_2$, respectively. The mutation rates were set to $\mutVec=(0.01,0.02,0.03)$ for all computations.
}
\label{fig:convergence}
\end{figure}

\begin{figure}
\centerline{
	\subfigure[]{\includegraphics[width=.49\textwidth]{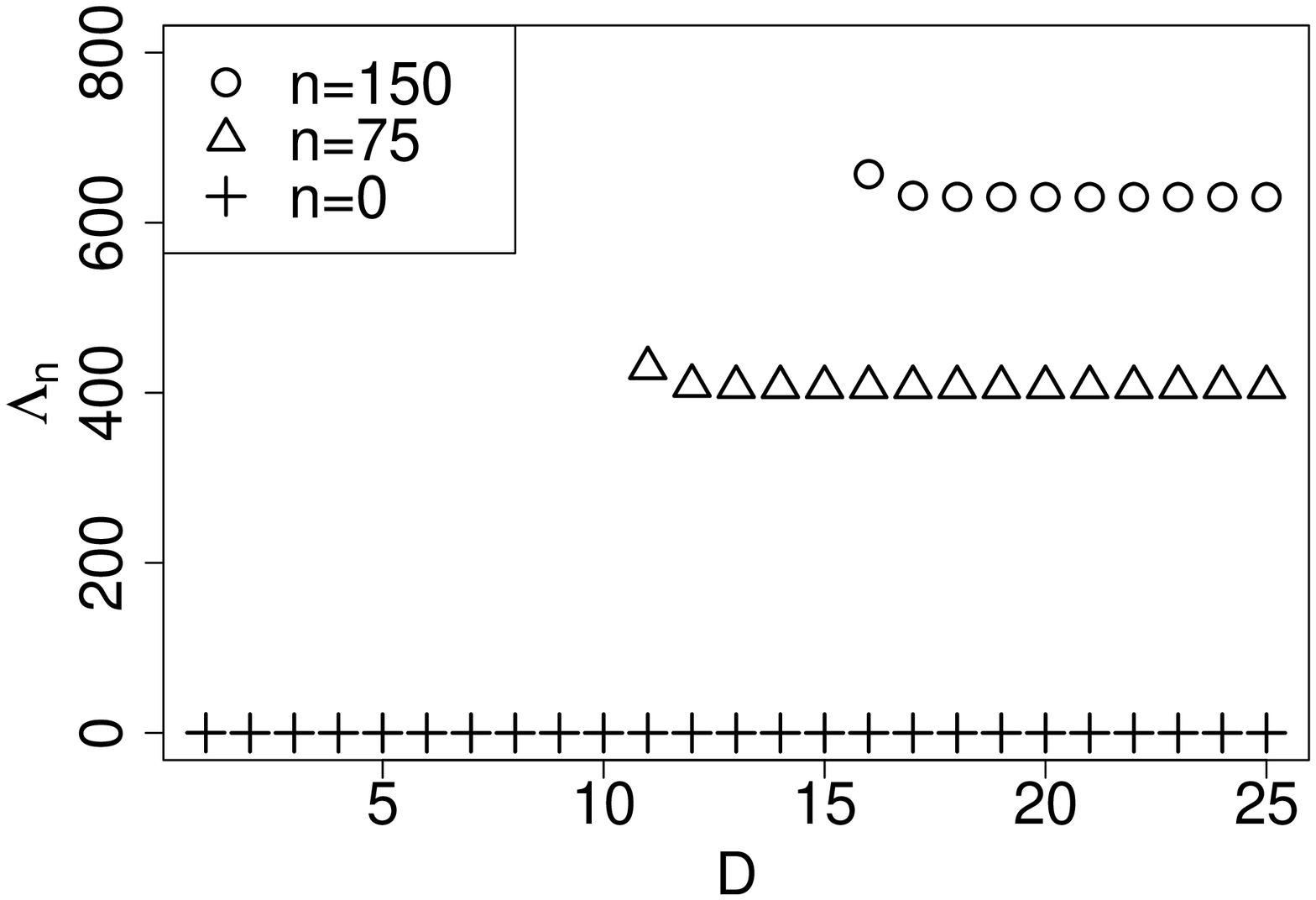}}
	\subfigure[]{\includegraphics[width=.49\textwidth]{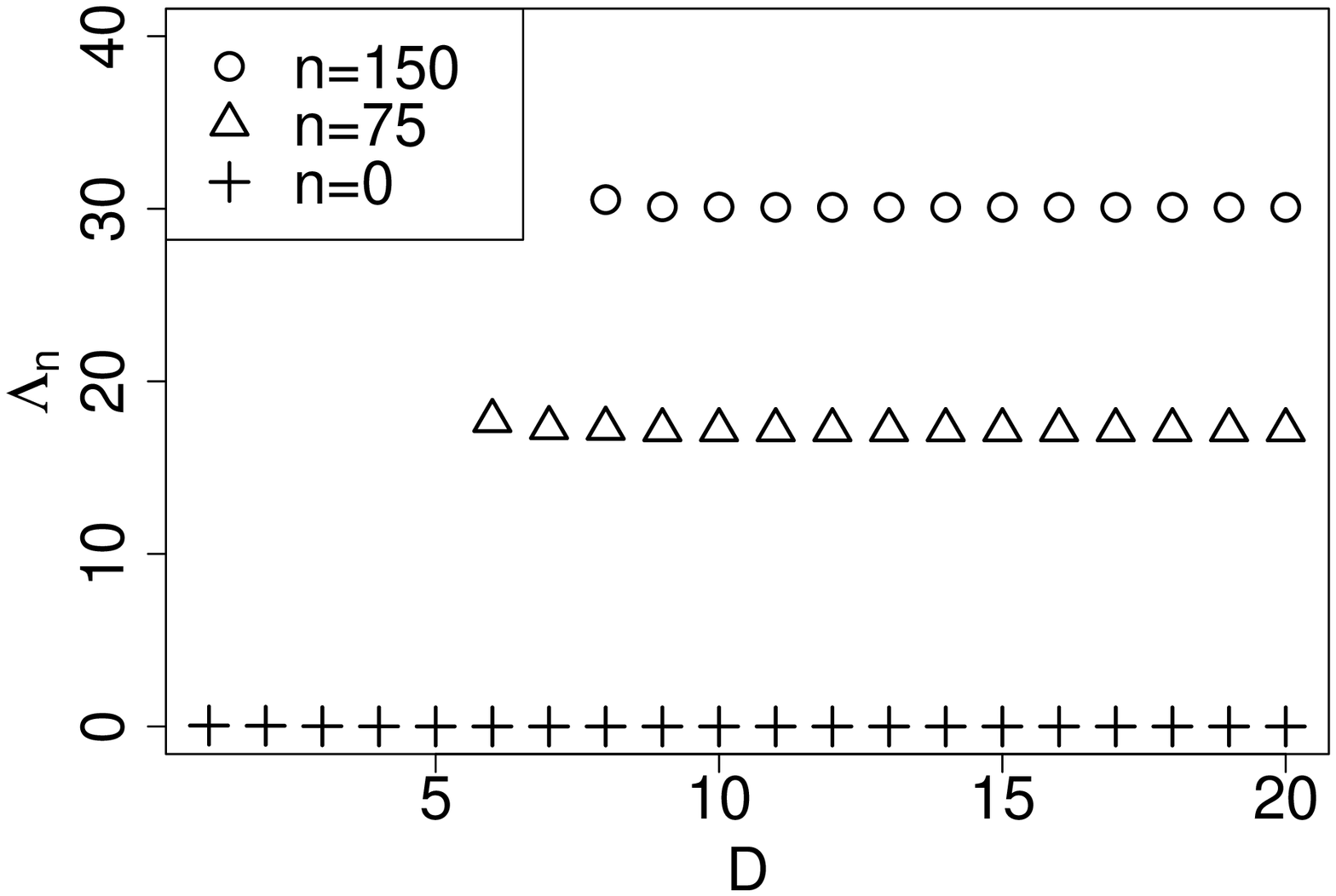}}
}
\centerline{
	\subfigure[]{\includegraphics[width=.49\textwidth]{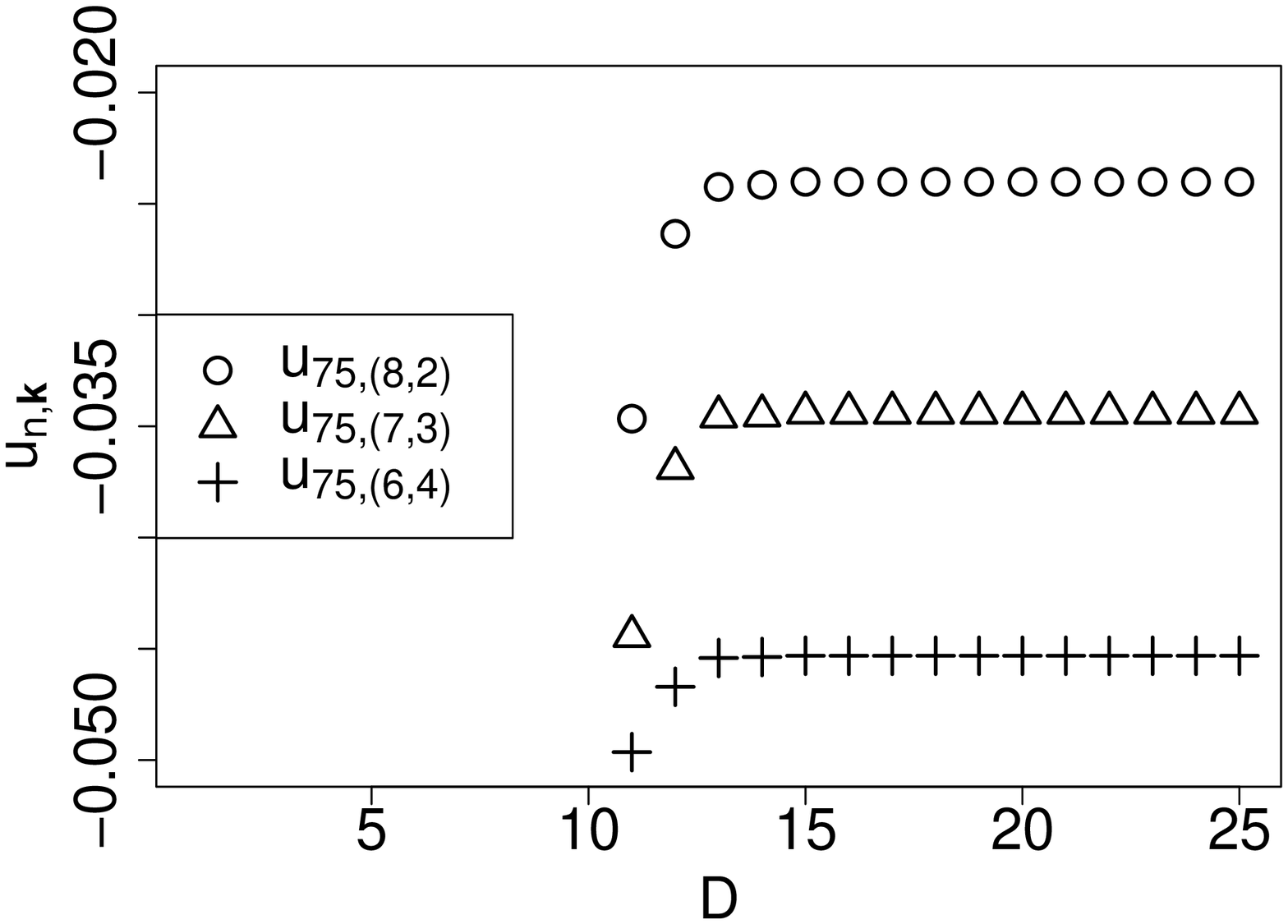}}
	\subfigure[]{\includegraphics[width=.49\textwidth]{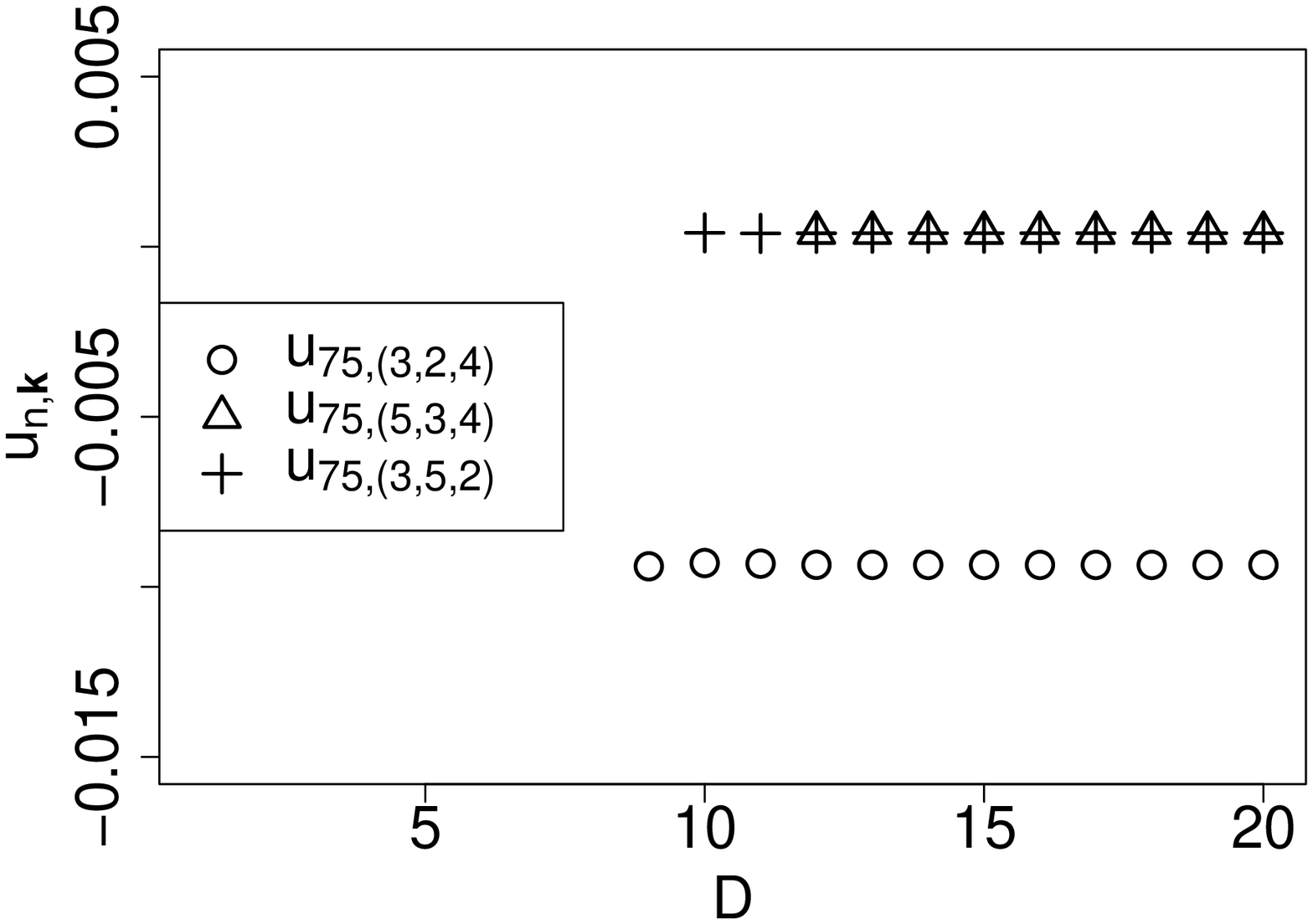}}
}
\caption{Convergence of the truncated eigenvalues $\eVal_{\eIdx}$ and coefficients of the eigenvectors $\eVec_{\eIdx}$ as the truncation level $D$ increases, for $\numAlleles=3$ with high mutation rates and for $\numAlleles=4$ with low mutation rates. Subfigures~(a) and~(c) show $\eVal_n^{[\matrixCutoff]}$ for $n=0,75,150$, and $\uCoeff_{75,\mIdx}^{[\matrixCutoff]}$ for $\mIdx=(8,2),(7,3),(6,4)$, respectively, for mutation rates $\mutVec=(10,20,30)$ and selection coefficients $\selMatrix=\selMatrix_1$. The convergence behavior for $\numAlleles=4$ is shown in subfigures~(b) and~(d) for $\eVal_n^{[\matrixCutoff]}$ with $n=0,75,150$, and $\uCoeff_{75,\mIdx}^{[\matrixCutoff]}$ with $\mIdx=(3,2,4),(5,3,4),(3,5,2)$, respectively;  the mutation rates were set to $\mutVec=(0.01,0.02,0.03,0.04)$ and the selection coefficients to $\selMatrix=\selMatrix_3$.}
\label{fig:convergence_more}
\end{figure}

Figure \ref{fig:evalues} shows $\eVal_\eIdx^{[\matrixCutoff]}$ for $\matrixCutoff=24$ and $0\leq \eIdx \leq 35$ under neutrality ($\selMatrix = \mathbf{0}$) and selection ($\selMatrix =\selMatrix_1$ and $\selMatrix =\frac{1}{4}\selMatrix_2$). Upon inspection all of the eigenvalues displayed have converged properly. Under neutrality, the eigenvalues are functions of $|\nIdx|$, thus they are degenerate and cluster into groups. In the presence of selection, however, we empirically observe that all of the eigenvalues are distinct. For moderate selection intensity, the group structure is less prominent. In general, increasing the selection parameters evens out the group structure and shifts the entire spectrum upward.

Computing the transition density function for large selection coefficients requires combining terms of substantially different orders of magnitude, because of the exponential weighting factors in the density~\eqref{eq_stat_dens_sel} and in the expansion~\eqref{eq:modJacobi}. Therefore, the coefficients $\uCoeff_{\eIdx,{\bf k}}^{[\matrixCutoff]}$ have to be calculated with high precision to obtain accurate numerical results under strong selection.

\begin{figure}
\centerline{
	\includegraphics[width=0.5\linewidth]{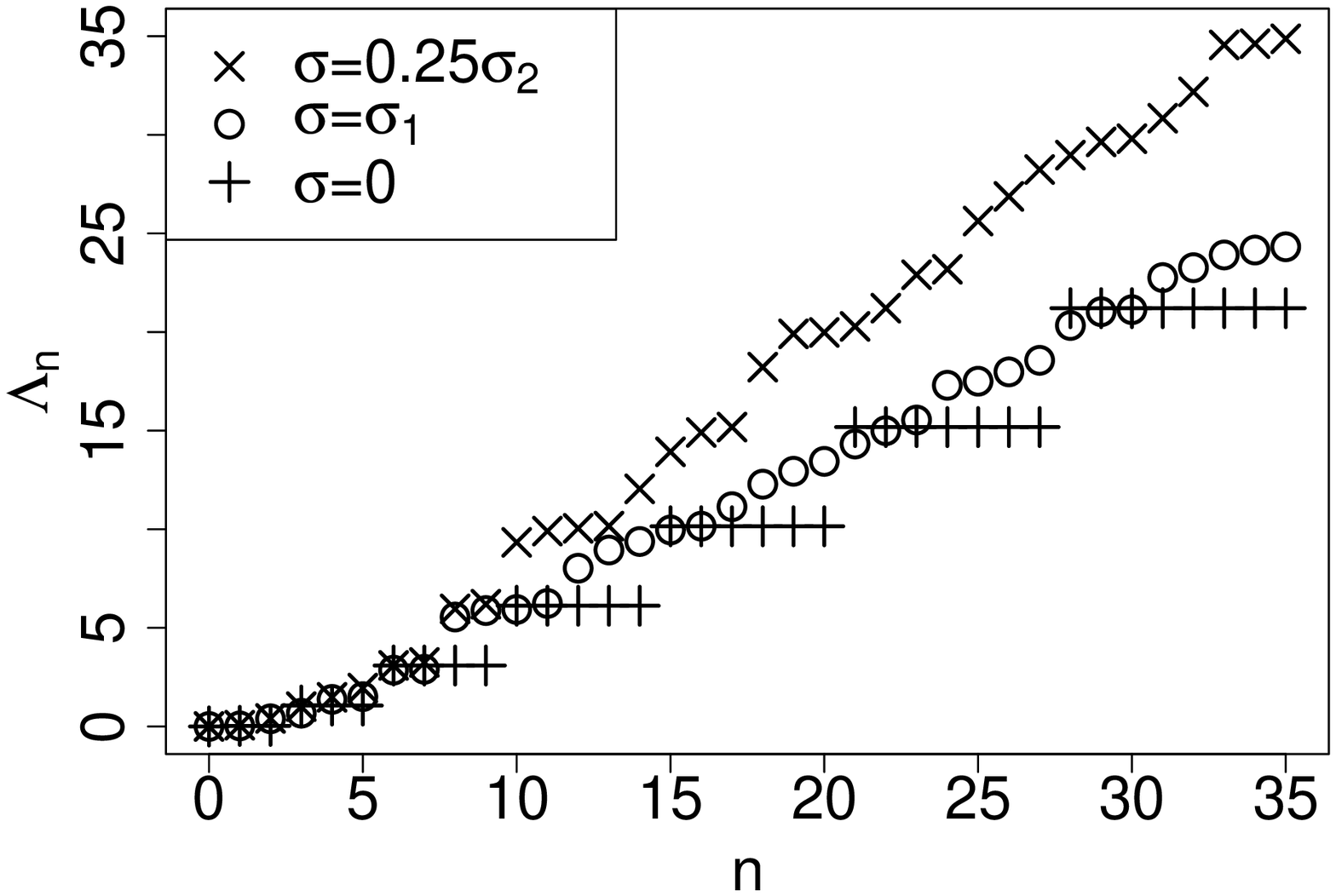}
}
\caption{
The first 36 eigenvalues of the different spectra for the selection parameters $\selMatrix=\boldsymbol{0}$, $\selMatrix_1$ and $\frac{1}{4}\selMatrix_2$, respectively.  The latter was chosen so that the ranges of the eigenvalues are comparable. The truncation level $\matrixCutoff$ was set to 24 and mutation rates $\mutVec=(0.01,0.02.0.03)$ were used.
}
\label{fig:evalues}
\end{figure}

\subsection{Transient and stationary densities}
\label{sec_stationary_distribution}

The approximations to the eigenvalues $\eVal_\eIdx^{[\matrixCutoff]}$ and the eigenfunctions $\eFun_\eIdx$ (via the eigenvectors $\eVec_\eIdx^{[\matrixCutoff]}$ and equation~\eqref{eq_linear_rep}) can be used in the spectral representation \eqref{eq_tdf} to approximate the transition density function at arbitrary times $t$. Examples with $\selMatrix = \selMatrix_1$ for different times are given in \fref{fig_tdf}. At first, the density is concentrated around the initial frequencies $\xVec = (0.02,0.02,0.96)$, but as time increases, the frequencies of the first and second allele increases, since these have a higher relative fitness. Eventually, the transition density converges to the stationary distribution (similar to distribution at $t=2$), where the bulk of the mass is concentrated at high frequencies for the first and second allele.

\begin{figure}
\centerline{
	\includegraphics[width=.49\textwidth]{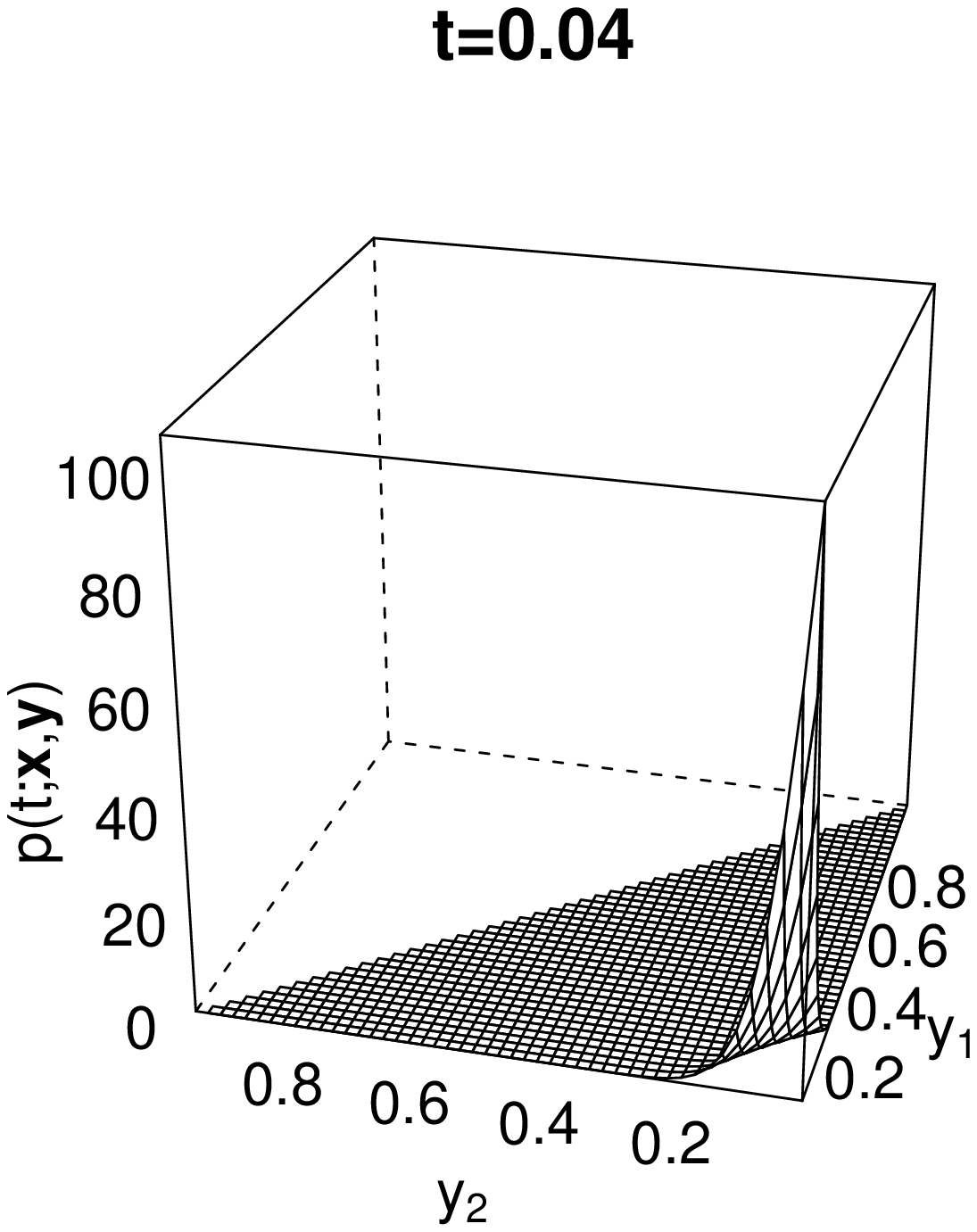}
	\includegraphics[width=.49\textwidth]{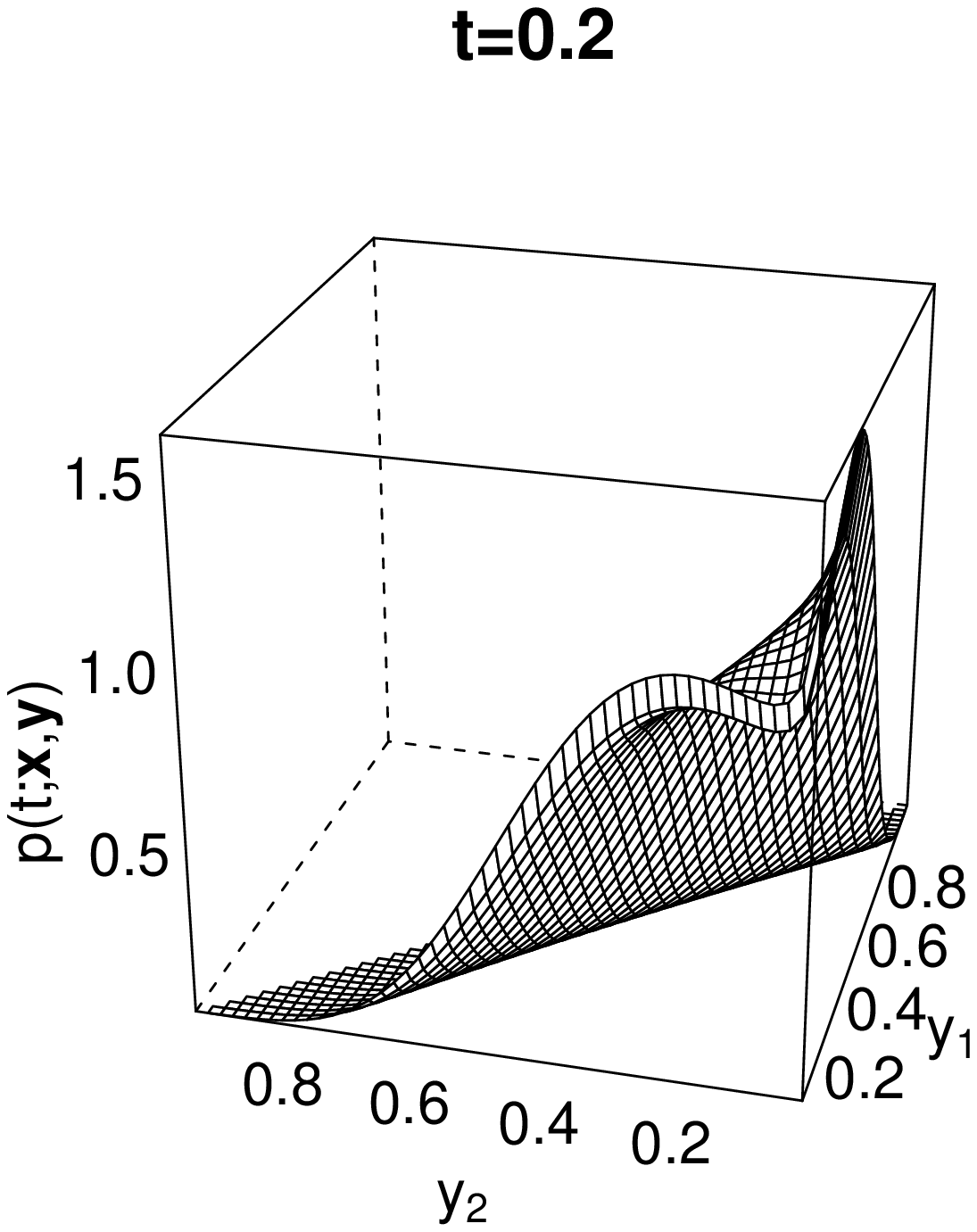}
}
\centerline{
	\includegraphics[width=.49\textwidth]{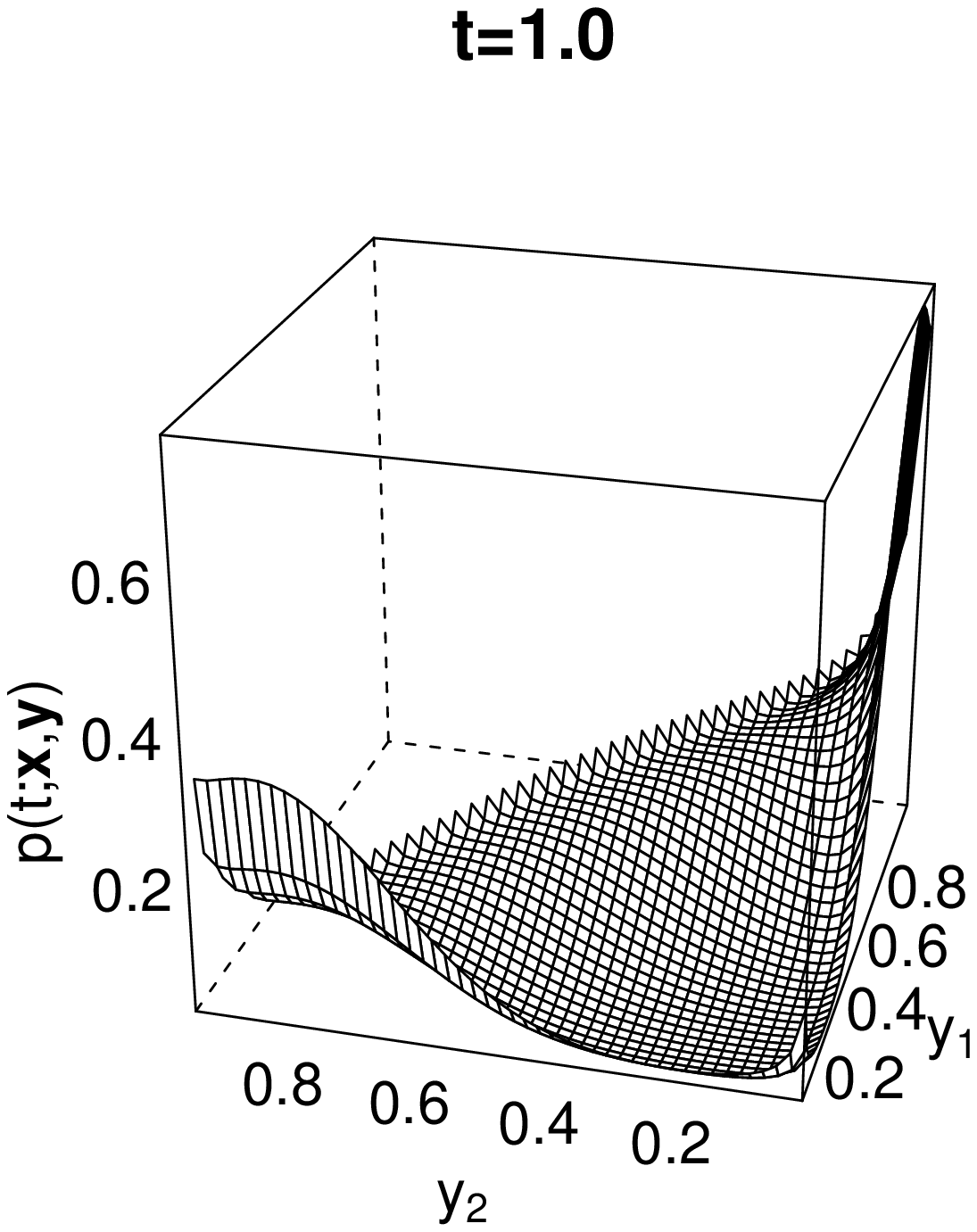}
	\includegraphics[width=.49\textwidth]{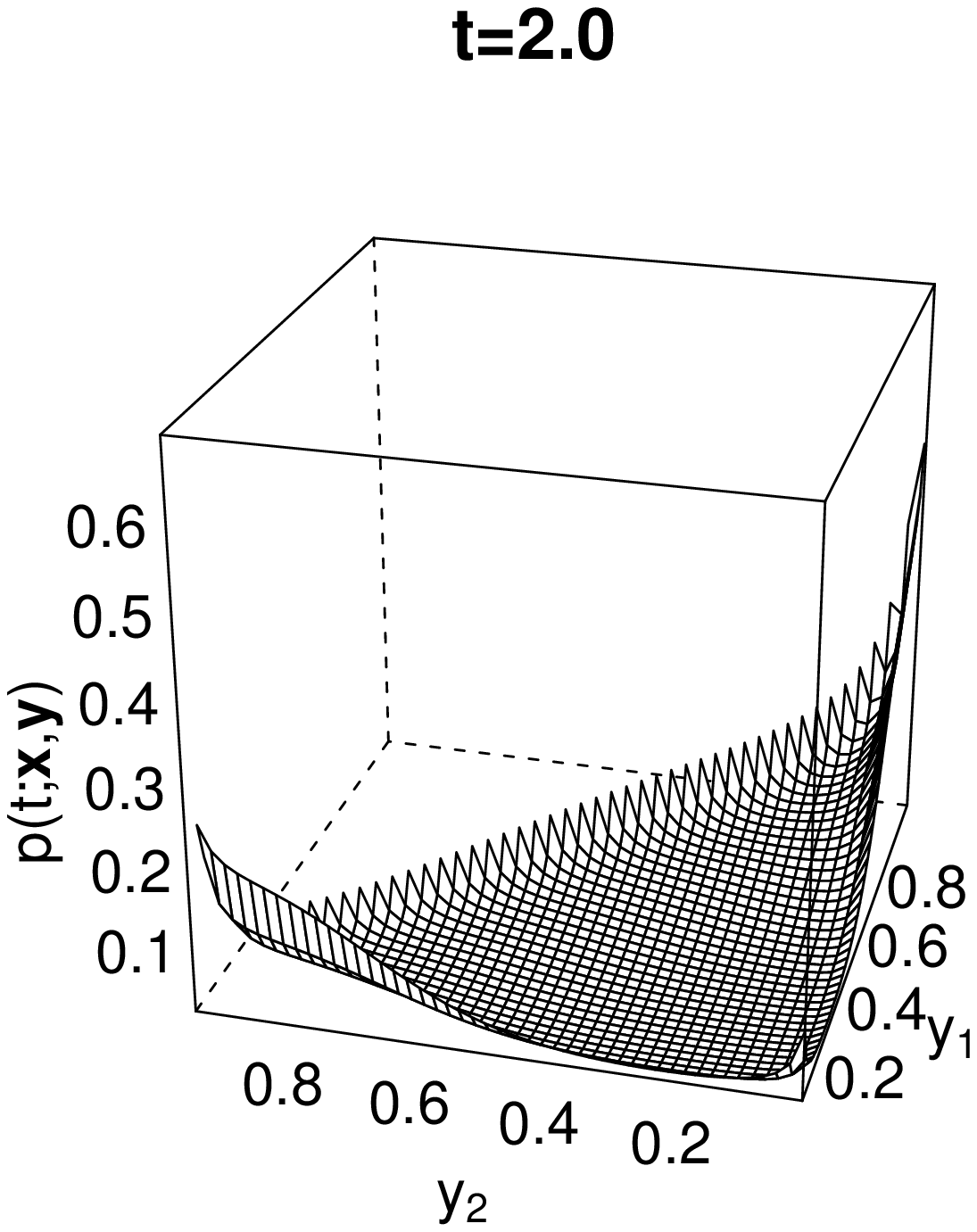}
}
\caption{Approximation of the transition density function~\eqref{eq_tdf} for different times $t \in \{0.04,0.2,1.0,2.0\}$. Selection was governed by the matrix of coefficients $\selMatrix = \selMatrix_1$ and $\xVec = (0.02,0.02,0.96)$ was used as initial condition. The truncation level was set to $\matrixCutoff = 40$, whereas the summation in equation~\eqref{eq_linear_rep} ranged over all $\mIdx$ such that $0 \leq \vert \mIdx \vert \leq 36$, and all eigenfunctions and eigenvalues with $0 \leq \eIdx \leq 561$ were included in equation~\eqref{eq_tdf}. The mutation rates were set to $\mutVec=(0.01,0.02.0.03)$. The plots only vary in $\yCmp_1$ and $\yCmp_2$, since $\yCmp_3 = 1 - \yCmp_1 -\yCmp_2$.}
\label{fig_tdf}
\end{figure}

The eigenvalues $\eVal_\eIdx^{[\matrixCutoff]}$ and coefficients $\eVec_{\eIdx,\mathbf{k}}^{[\matrixCutoff]}$ can also be employed to approximate the constant that normalizes the stationary distribution $\stat(\xVec)$ to a proper probability distribution. Following the same line of argument as~\citet{Song2012}, the orthogonal relations enable us to circumvent the difficulty involved in directly evaluating a multivariate integral over the simplex $\simplex_{\numAlleles-1}$. First, note that since $\genFull$ maps constant functions to zero, any constant function is an eigenfunction with associated eigenvalue $\eVal_0=0$, thus $\eFun_0(\xVec) = \eFun_0(\yVec) = \text{const}$.  In \eqref{eq_tdf}, taking $t\to\infty$, we get
\begin{equation}\label{eq_stat_distr}
\lim_{t\to\infty}p(t;\xVec,\yVec)=\stat(\yVec)\frac{\eFun_0(\xVec)\eFun_0(\yVec)}{\langle \eFun_0,\eFun_0\rangle_{\stat}} =: \frac{1}{\normConst}\stat(\yVec).
\end{equation}
Then for $\xVec=\yVec={\bf 0}$, by \eqref{eq_linear_rep} we have
\begin{equation}\label{eq_normConst_with_sel}
	\begin{split}
		\normConst & = \int_{\simplex_{\numAlleles-1}} \stat({\bf z}) d{\bf z} = \frac{\langle \eFun_0,\eFun_0\rangle_{\stat}}{\eFun_0({\bf 0})^2} \\
			& = \frac{ \sum_{\mIdx\in\N_0^{\numAlleles-1}} \uCoeff_{0,\mIdx}^2 \langle \jacobi_\mIdx^{\mutVec} , \jacobi_\mIdx^{\mutVec}\rangle _{\statNeutr}}{ e^{-\meanFitness({\bf 0})} \left( \sum_{\mIdx\in\N_0^{\numAlleles-1}} \uCoeff_{0,\mIdx} \jacobi_\mIdx^{\mutVec}({\bf 0}) \right)^2}\\
			& = \frac{ \sum_{\mIdx\in\N_0^{\numAlleles-1}} \uCoeff_{0,\mIdx}^2 \jacLength_\mIdx^{\mutVec} }{ \left(\sum_{\mIdx\in\N_0^{\numAlleles-1}} \uCoeff_{0,\mIdx} \prod_{j=1}^{\numAlleles-1} \frac{\Gamma(\nIdxCmp_j+\mutCmp_j)}{\Gamma(\nIdxCmp_j+1)\Gamma(\mutCmp_j)} \right)^2 },
	\end{split}
\end{equation}
since $\meanFitness({\bf 0}) = \selCmp_{\numAlleles,\numAlleles} = 0$ and $\uJacobi_{\nIdxCmp_j}^{(\mutCmp_j,\mutTot_j+2\nTot_j)}(0)=(-1)^{\nIdxCmp_j}\frac{\Gamma(\nIdxCmp_j+\mutCmp_j)}{\Gamma(\nIdxCmp_j+1)\Gamma(\mutCmp_j)}$.  Here $\jacLength_\mIdx^{\mutVec}$ is the constant defined in \eqref{eq_const}. The purely algebraic form of the right hand side in equation~\eqref{eq_normConst_with_sel} allows to compute an accurate approximation of the normalizing constant $\normConst$ by replacing the infinite sums by sums over all indices $\mIdx$ such that $\vert \mIdx \vert$ is less or equal then a given truncation level. This offers an attractive alternative to other computationally intensive methods \citep{Donnelly2001,Genz2003,Buzbas2009}.

\fref{fig_stat} shows two examples of stationary distributions for different selection coefficients. In \fref{fig_statA} the  stationary density is concentrated in the interior of the simplex, since all homozygotes are less fit then the heterozygotes. This situation is referred to as heterozygote advantage, resulting in a balancing selection pattern, and the different alleles co-exist at stationarity. In \fref{fig_statB}, allele number 1 is strongly favored by the given selection coefficients, and thus the stationary density is concentrated at high frequencies for this allele.

\begin{figure}
\centerline{
	\subfigure[Selection coefficients][Selection coefficients: $\selMatrix = \left(\begin{matrix}0&15&15\\15&0&15\\15&15&0\end{matrix}\right)$\label{fig_statA}]{\includegraphics[width=.49\textwidth]{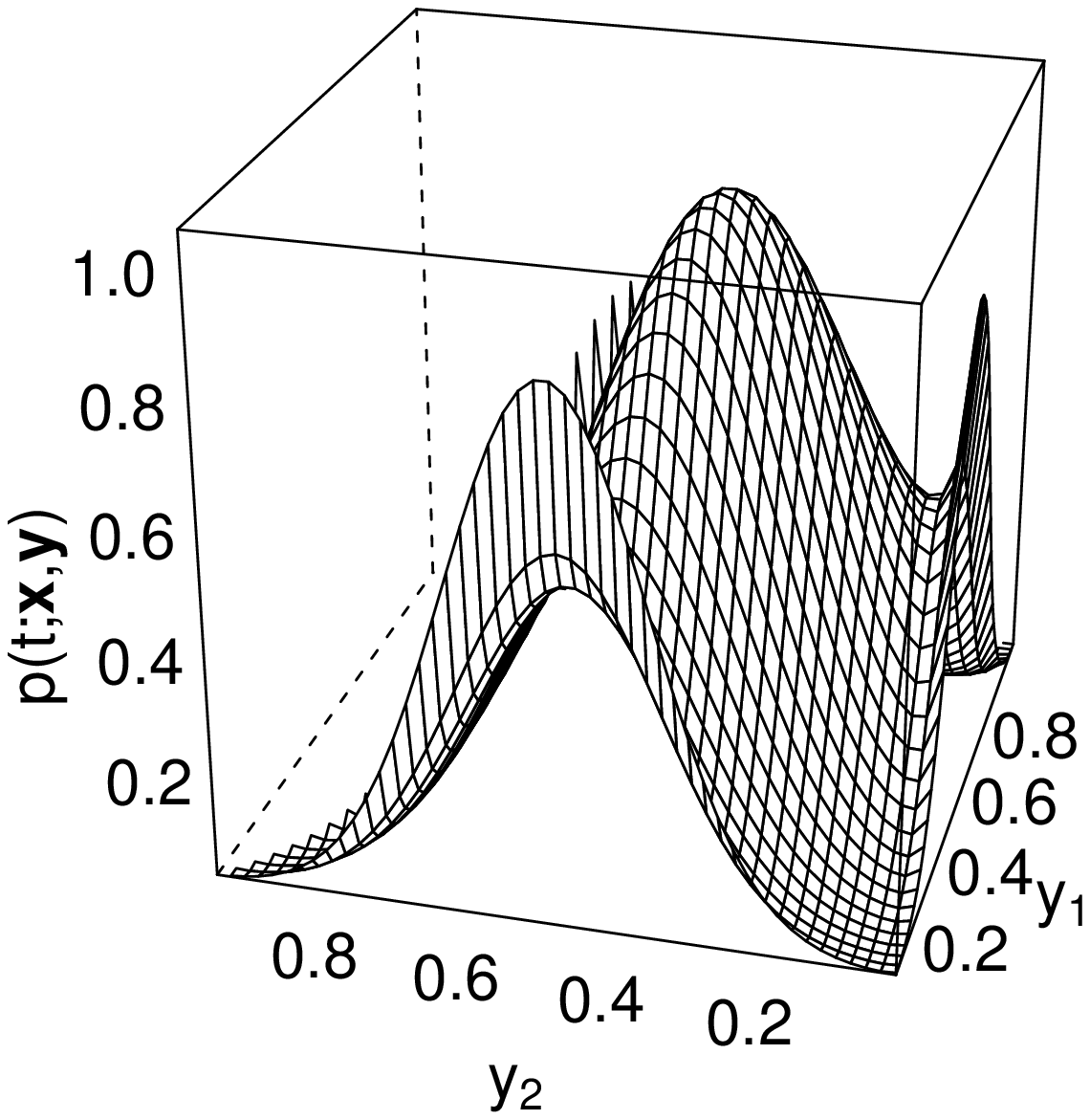}}
	\subfigure[Selection coefficients][Selection coefficients: $\selMatrix = \left(\begin{matrix}10&6&-5\\6&-6&6\\-5&6&0\end{matrix}\right)$\label{fig_statB}]{\includegraphics[width=.49\textwidth]{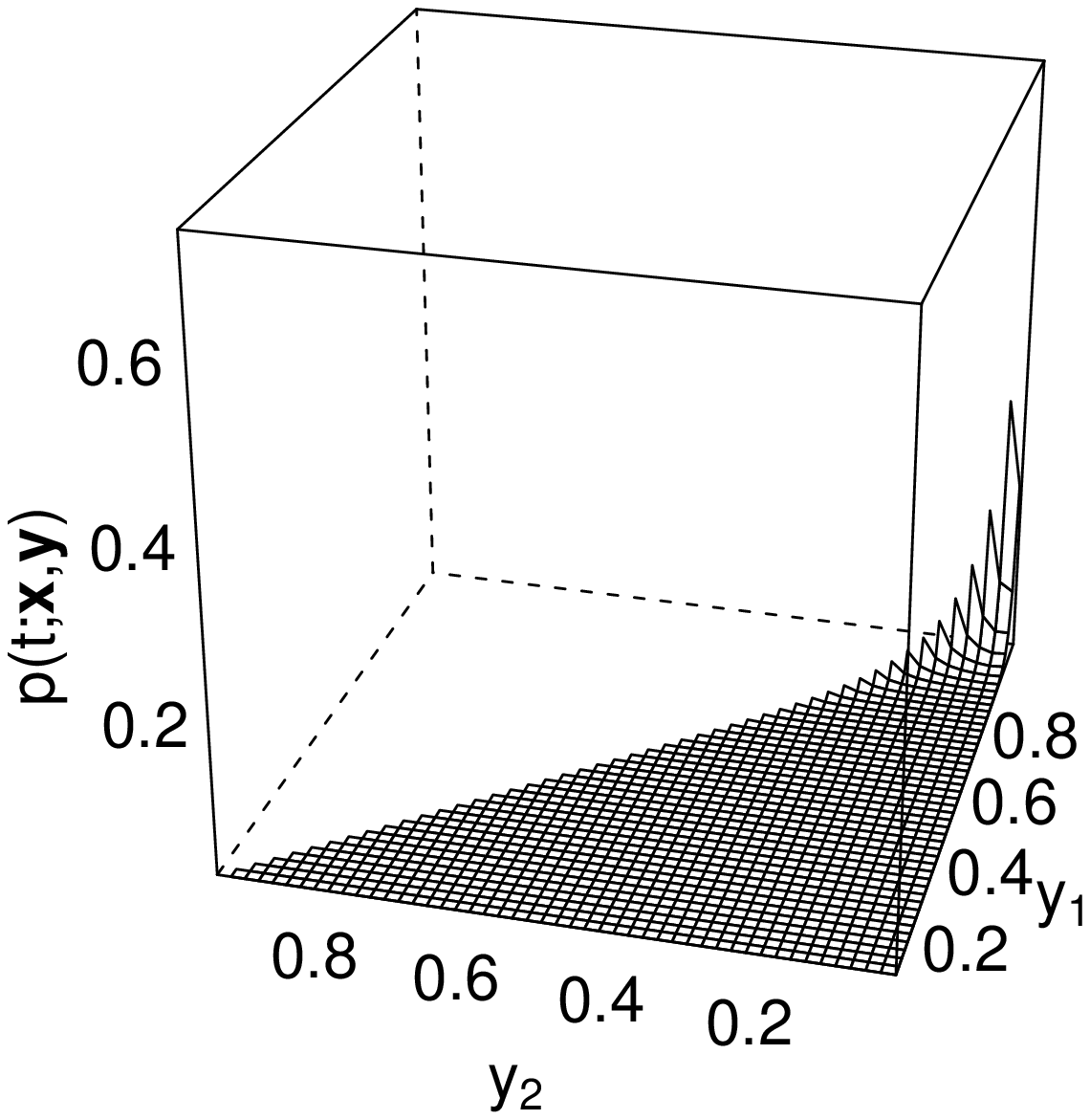}}
}
\caption{Two examples of the stationary distribution for different selection parameters. The mutation rates were set to $\mutVec=(0.01,0.02.0.03)$ in both cases. Again, a truncation level of $\matrixCutoff = 40$ was used, the summation in equation~\eqref{eq_normConst_with_sel} ranged over all $\mIdx$ such that $0 \leq \vert \mIdx \vert \leq 36$.  The plots only vary in $\yCmp_1$ and $\yCmp_2$, since $\yCmp_3 = 1 - \yCmp_1 -\yCmp_2$.}
\label{fig_stat}
\end{figure}

We can also use \eqref{eq_tdf} to investigate the rate of convergence of the diffusion process to the stationary distribution. Denote the difference between the transition density and the stationary density by
\begin{eqnarray*}
d(t;\xVec,\yVec) & := & p(t;\xVec,\yVec) - \frac{1}{\normConst}\stat(\yVec)\\
 & = & \sum_{\eIdx=1}^{\infty}e^{-\eVal_\eIdx t}\stat(\yVec)\frac{\eFun_\eIdx(\xVec)\eFun_\eIdx(\yVec)}{\langle \eFun_\eIdx,\eFun_\eIdx\rangle _{\stat}}.
\end{eqnarray*}
We measure the magnitude of $d(t;\xVec,\yVec)$ by the square of its $L^2$ norm with respect to the weight function $1/\stat(\yVec)$, that is,
\begin{equation}\label{eq_dev_from_stat}
	\begin{split}
		\left\Vert d(t;\xVec,\cdot) \right\Vert_{1/\stat}^2 & := \langle d,d\rangle _{1/\stat} = \sum_{\eIdx=1}^{\infty}e^{-2\eVal_\eIdx t}\frac{\eFun_\eIdx(\xVec)^2}{\langle \eFun_\eIdx,\eFun_\eIdx\rangle _{\stat}} \\
			& = \sum_{\eIdx=1}^{\infty}e^{-2\eVal_\eIdx t}\frac{e^{-\meanFitness(\xVec)}\left(\sum_{{\bf k}\in\N_0^{\numAlleles-1}}\uCoeff_{\eIdx,{\bf k}}\jacobi_{\bf k}^{\mutVec}(\xVec)\right)^2}{\sum_{\mIdx\in\N_0^{\numAlleles-1}}u^2_{\eIdx,\mIdx}\jacLength_\mIdx^{\mutVec}}.
	\end{split}
\end{equation}
Again, the sums in this expression can be approximated by truncating at a given level. Figure~\ref{fig_density_cvg} shows  $\left\Vert d(t;\xVec,\cdot)\right\Vert^2_{1/\stat}$ as a function of time $t$, for $\selMatrix=\selMatrix_1$, $\selMatrix=0.5\selMatrix_1$, $\selMatrix=0.1\selMatrix_1$. The initial frequencies were $\xVec=(0.02,0.02,0.96)$. As expected, the distance to the stationary distribution decreases over time. Further, the rate of convergence is faster if the values in $\selMatrix$ get larger, which was observed by \citet{Song2012} too. We note that the spectral representation can also be readily employed to study convergence rates measured by other metrics such as the total variation distance or relative entropy.

\begin{figure}
\centerline{
	\includegraphics[width=0.5\linewidth]{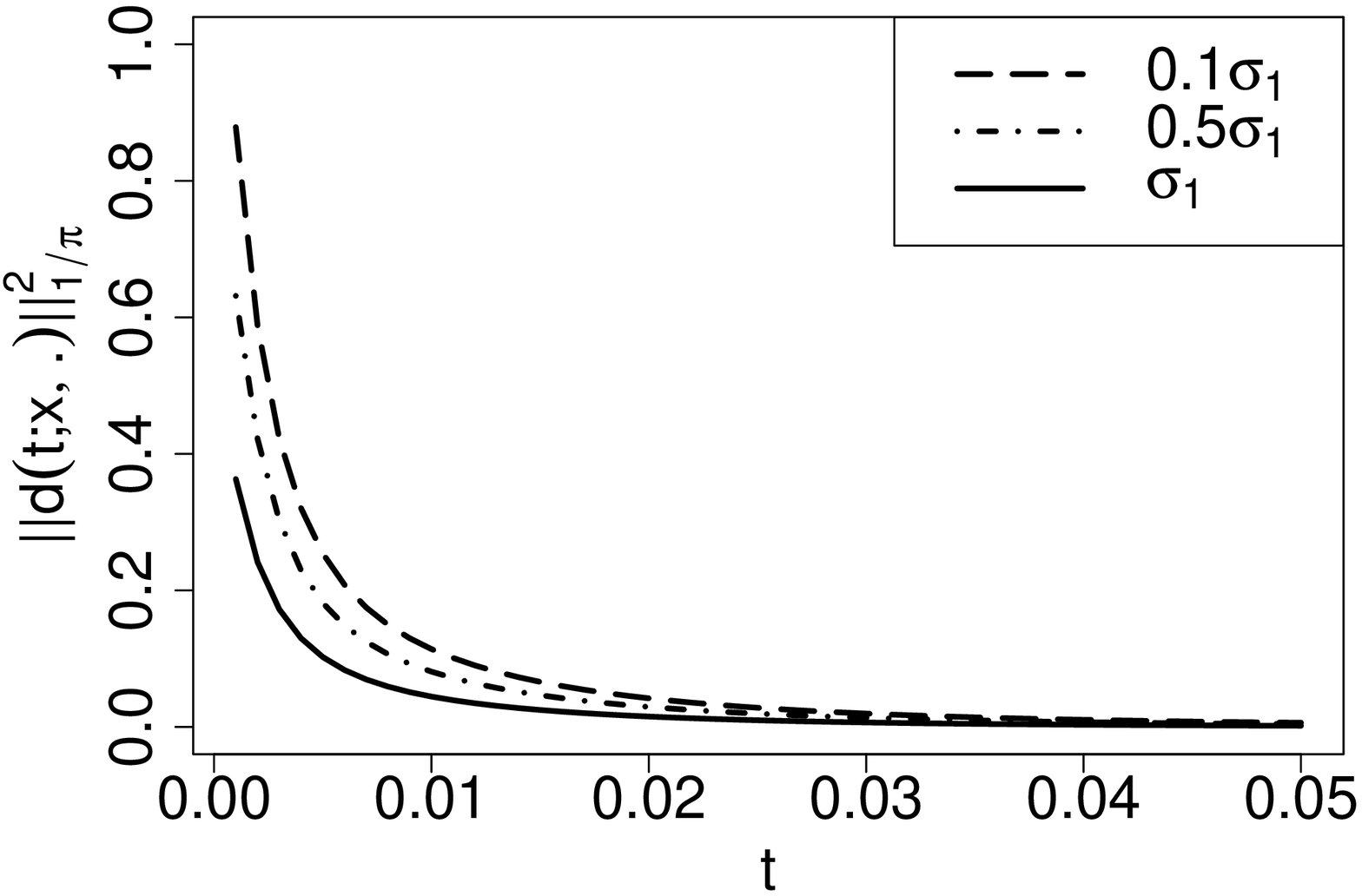}
}
\caption{Convergence of the transition density to stationarity as time evolves, for initial frequencies $\xVec=(0.02,0.02,0.96)^T$. Deviation from the stationary density is measured by $\left\Vert d(t;\xVec,\cdot)\right\Vert^2_{1/\stat}$, defined in~\eqref{eq_dev_from_stat}. The mutation rates were chosen to be $\mutVec=(0.01,0.02,0.03)$ and the selection parameters were $\selMatrix=0.1\selMatrix_1$, $\selMatrix=0.5\selMatrix_1$ and $\selMatrix=0.1\selMatrix_1$, respectively. The truncation level was set to $D=40$, and~\eqref{eq_dev_from_stat} was approximated by summing over $0\leq \eIdx \leq 561$ and $\mIdx,\mathbf{k}$ such that $0 \leq \vert \mIdx \vert, \vert \mathbf{k} \vert \leq 36$.} 
\label{fig_density_cvg}
\end{figure}

\section{Discussion}
\label{sec_discussion}

In this paper, we have extended the method of \citet{Song2012} to obtain an explicit spectral representation of the transition density function for the multi-dimensional Wright-Fisher diffusion under a PIM model with general diploid selection and an arbitrary number of alleles.  We have demonstrated the fidelity and fast convergence of the approximations.  Further, as an example application of our work, we have computed the normalization constant of the stationary distribution and quantified the rate at which the transition density approaches this distribution.

Efficient approximations of the eigensystem and the transition density function lead to a number of important applications. Combining the stationary distribution discussed in \sref{sec_stationary_distribution} with the recurrence relation shown in Lemma~\ref{lem_mult_recurrence}, one can calculate algebraically the probability of observing a given genetic configuration of individuals sampled from the stationary distribution of the non-neutral diffusion.  
This kind of algebraic approach would complement previous works \citep{Evans2007,Zivkovic2011} on sample allele frequency spectra that involve solving ODEs satisfied by the moments of the diffusion.
Further, the algebraic approach is potentially more efficient than computationally expensive Monte Carlo methods \citep{Donnelly2001} and more generally applicable than methods relying on the selection coefficients being of a certain form \citep{Genz2003}. Note that, by discretizing time and space, our representation of the transition density function can be used for approximate simulation of frequencies from stationarity as well as frequency trajectories, which can in turn be employed in the aforementioned Monte Carlo frameworks.

The sampling probability can be applied, for example, to estimate evolutionary parameters via maximum likelihood or Bayesian inference frameworks. Furthermore, the  notion of sampling probability can be combined with the spectral representation of the transition density function in a hidden Markov model framework as in \citet{Bollback2008}, to calculate the probability of observing a series of configurations sampled at different times. The method developed in this paper would allow for such an analysis in a model with multiple alleles subject to recurrent parent-independent mutation and general diploid selection. 

An important, albeit very challenging, future direction is to extend our current approach to analyze the dynamics of multi-locus diffusions with recombination and selection.  Such an extension would allow for the incorporation of additional data at closely linked loci, which has the potential to significantly improve the inference of evolutionary parameters, especially the strength and localization of selection.  
We have only considered Wright-Fisher diffusions in a single panmictic population of a constant size. As achieved in the alternative approaches of \citet{Gutenkunst2009} and \citet{Lukic2011}, mentioned in Introduction, it would be desirable to generalize our approach to incorporate subdivided populations exchanging migrants, with possibly fluctuating population sizes.  Another possible extension is to relax the PIM assumption and consider a more general mutation model.

We note that our present technique relies on the diffusion generator being symmetric.  This symmetry does not hold in some of the scenarios mentioned above, making a direct application of the ideas developed here difficult.  However, we believe that it is worthwhile investigating whether one could apply our approach to devise approximations to the transition density function that are sufficiently accurate for practical applications.

\section*{Acknowledgement}

We thank Anand Bhaskar for many helpful discussions. 
This research is supported in part by a DFG Research Fellowship STE 2011/1-1 to M.S.; and by an NIH grant R01-GM094402, an Alfred P. Sloan Research Fellowship, and a Packard Fellowship for Science and Engineering to Y.S.S.

\appendix

\section{Proofs of lemmas}
\label{app_proofs}

\begin{proof}[Proof of Lemma~\ref{lem_orthogonality}]
For two indices $\nIdx,\mIdx \in \N_0^{\numAlleles-1}$, consider the integral
\begin{equation}\label{eq_orthogonality}
	\begin{split}
		\int_{\simplex_{\numAlleles-1}}  \jacobi^\mutVec_\nIdx (\xVec) \jacobi^\mutVec_\mIdx (\xVec) \statNeutr(\xVec) d\xVec & = \int_{[0,1]^{\numAlleles-1}}  \jacobiTr^\mutVec_\nIdx (\xVecTr) \jacobiTr^\mutVec_\mIdx (\xVecTr) \statNeutr\big( \xVec(\xVecTr)\big) \big|\mathrm{det}(D\xVec)(\xVecTr)\big|d\xVecTr,
	\end{split}
\end{equation}
where the right hand side can be obtained by the coordinate transformation introduced in Appendix~\ref{app_coord_transform} and the multivariate integration through substitution rule. Using
\begin{equation}\label{eq_weight_function_transformed}
	\begin{split}
		\statNeutr\big( \xVec(\xVecTr)\big) = \prod_{i=1}^{\numAlleles-1} \xCmpTr_i^{\mutCmp_i -1} (1 -\xCmpTr_i)^{\mutTot_i - (\numAlleles - i)},\\
	\end{split}
\end{equation}
the determinant of the Jacobian 
\begin{equation}
	\begin{split}
		\big|\mathrm{det}(D\xVec)(\xVecTr)\big| = \prod_{i=1}^{\numAlleles-2}(1-\xCmpTr_i)^{\numAlleles-(i+1)},
	\end{split}
\end{equation}
see~\cite{Baxter2007}[Equation~B.1], and the transformed Jacobi polynomials~\eqref{eq_transformed_jacobis}, it can be shown that
\begin{equation}
	\begin{split}
\prod_{j=1}^{\numAlleles-1} & \int_0^1 \uJacobi_{\nIdxCmp_j}^{(\mutCmp_j,\mutTot_j + 2\nTot_j)}(\xCmpTr_j) \uJacobi_{\mIdxCmp_j}^{(\mutCmp_j,\mutTot_j + 2\mTot_j)}(\xCmpTr_j) \xCmpTr_j^{\mutCmp_j -1} (1 -\xCmpTr_j)^{\mutTot_j + \nTot_j + \mTot_j- 1} d\xCmpTr_j = \jacLength^\mutVec_\nIdx \delta_{\nIdx,\mIdx}
	\end{split}
\end{equation}
holds, with
\begin{equation}
	\jacLength^\mutVec_\nIdx = \prod_{j=1}^{\numAlleles-1} \uJacLength_{\nIdxCmp_j}^{(\mutCmp_j,\mutTot_j + 2\nTot_j)}.
\end{equation}
In the case $\nIdx=\mIdx$ this can be seen immediately. If $\nIdx\neq \mIdx$ without loss of generality let $1 \leq l \leq \numAlleles-1$ be the largest $l$ such that $\nIdxCmp_l < \mIdxCmp_l$ and $\nIdxCmp_k = \mIdxCmp_k$ for all $k = l+1, \ldots, \numAlleles-1$. Then $\nTot_l = \mTot_l$ (recall $\nTot_{\numAlleles-1} = \mTot_{\numAlleles-1} = 0$) and $\uJacobi^{(\mutCmp_l,\mutTot_l + 2\mTot_l)}_{\mIdxCmp_l}(\xCmpTr_l)$ is orthogonal to all polynomials of lesser degree with respect to the weight function $\xCmpTr_l^{\mutCmp_l -1} (1 -\xCmpTr_l)^{\mutTot_l + 2\mTot_l- 1}$, and thus the $l$-th factor and the whole product is zero.
\end{proof}

\begin{proof}[Proof of Lemma~\ref{lem_mult_recurrence}]
We found it most convenient to derive a recurrence relation for
\begin{equation}\label{eq_times_x}
	\xCmp_i \jacobi_\nIdx^\mutVec (\xVec)
\end{equation}
by projecting expression~\eqref{eq_times_x} onto the orthogonal basis $\big\{ \jacobi_\mIdx^\mutVec(\xVec) \big\}$, and investigate the respective coefficients. 
First, note that the coordinate transformation introduced in Appendix~\ref{app_coord_transform} yields $\xCmp_i = \xCmpTr_i \prod_{j<i} (1 - \xCmpTr_j)$, so
\begin{equation}\label{eq_transformed_monomial}
	\xCmp_i \jacobi_\nIdx^\mutVec (\xVec) =  \xCmpTr_i \prod_{j<i} (1 - \xCmpTr_j) \jacobiTr^\mutVec_\nIdx (\xVecTr).
\end{equation}
Further, integrate expression~\eqref{eq_transformed_monomial} against the base function $\jacobiTr_\mIdx^\mutVec (\xVecTr)$ times the weight function $\statNeutr$ to get the respective coefficient in the basis representation. Using the integration by substitution rule again, as in equation~\eqref{eq_orthogonality}, this yields
\begin{equation}\label{eq_projection}
	\begin{split}
		\frac{1}{\jacLength_\mIdx^\mutVec} & \int_{[0,1]^{\numAlleles-1}} \xCmpTr_i \prod_{j<i} (1 - \xCmpTr_j) \jacobiTr^\mutVec_\nIdx (\xVecTr) \jacobiTr^\mutVec_\mIdx (\xVecTr) \prod_{k=1}^{\numAlleles-1} \xCmpTr_k^{\mutCmp_k -1} (1 - \xCmpTr_k)^{\mutTot_k - 1} d\xVecTr\\
		& = \prod_{j=i+1}^{\numAlleles-1} \frac{1}{\uJacLength_{\mIdxCmp_j}^{(\mutCmp_j,\mutTot_j + 2\mTot_j)}} \int_0^1 \uJacobi_{\nIdxCmp_j}^{(\mutCmp_j,\mutTot_j + 2\nTot_j)}(\xCmpTr_j) \uJacobi_{\mIdxCmp_j}^{(\mutCmp_j,\mutTot_j + 2\mTot_j)}(\xCmpTr_j) \xCmpTr_j^{\mutCmp_j -1} (1 -\xCmpTr_j)^{\mutTot_j + \nTot_j + \mTot_j- 1} d\xCmpTr_j\\
		& \qquad \times \frac{1}{\uJacLength_{\mIdxCmp_i}^{(\mutCmp_i,\mutTot_i + 2\mTot_i)}} \int_0^1 \uJacobi_{\nIdxCmp_i}^{(\mutCmp_i,\mutTot_i + 2\nTot_i)}(\xCmpTr_i) \uJacobi_{\mIdxCmp_i}^{(\mutCmp_i,\mutTot_i + 2\mTot_i)}(\xCmpTr_i) \xCmpTr_i^{\mutCmp_i} (1 - \xCmpTr_i)^{\mutTot_i + \nTot_i + \mTot_i- 1} d\xCmpTr_i\\
		& \qquad \times \prod_{j=1}^{i-1} \frac{1}{\uJacLength_{\mIdxCmp_j}^{(\mutCmp_j,\mutTot_j + 2\mTot_j)}} \int_0^1 \uJacobi_{\nIdxCmp_j}^{(\mutCmp_j,\mutTot_j + 2\nTot_j)}(\xCmpTr_j) \uJacobi_{\mIdxCmp_j}^{(\mutCmp_j,\mutTot_j + 2\mTot_j)}(\xCmpTr_j) \xCmpTr_j^{\mutCmp_j -1} (1 -\xCmpTr_j)^{\mutTot_j + \nTot_j + \mTot_j} d\xCmpTr_j.\\
	\end{split}
\end{equation}

The first term on the right hand side yields zero, unless $\mIdxCmp_j = \nIdxCmp_j$ for all $j>i$, thus $\mTot_i = \nTot_i$. In this case the term is equal to 1. Since $\mIdxCmp_j = \nIdxCmp_j$ for $j>i$, note that the second term on the right hand side is of the form
\begin{equation}\label{integral_weight_alpha_plus}
	\begin{split}
		\frac{1}{\uJacLength_{\mIdxCmp_i}^{(\alpha, \beta)}} & \int_0^1 \uJacobi_{\nIdxCmp_i}^{(\alpha,\beta)}(\xCmpTr) \uJacobi_{\mIdxCmp_i}^{(\alpha,\beta)}(\xCmpTr) \, \xCmpTr \, w_{\alpha,\beta}(\xCmpTr) d\xCmpTr\\
			& = G^{(\alpha,\beta)}_{\nIdxCmp_i,\mIdxCmp_i} \delta_{\nIdxCmp_i+1,\mIdxCmp_i} + G^{(\alpha,\beta)}_{\nIdxCmp_i,\mIdxCmp_i} \delta_{\nIdxCmp_i,\mIdxCmp_i} + G^{(\alpha,\beta)}_{\nIdxCmp_i,\mIdxCmp_i} \delta_{\nIdxCmp_i-1,\mIdxCmp_i},
	\end{split}
\end{equation}
with $w_{\alpha,\beta}(\xCmpTr) = \xCmpTr^{\alpha-1}(1-\xCmpTr)^{\beta-1}$, $\alpha = \mutCmp_i$, and $\beta = \mutTot_i + 2\nTot_i$. Here we applied the recurrence relation~\eqref{eq:R_recurrence} to $ \xCmpTr \uJacobi_{\nIdxCmp_i}^{(\alpha,\beta)}(\xCmpTr)$ and used the orthogonality of the Jacobi polynomials. The constants $G^{(a,b)}_{n,m}$ are given by
\[
G_{n,m}^{(a,b)} = \begin{cases}
	\frac{(n + a - 1) (n + b - 1)}{(2 n + a + b - 1) (2 n + a + b - 2)},	& \text{if $n-m=1$ and $n > 0$},\\
    \frac{1}{2} - \frac{b^2 - a^2 - 2 (b - a)}{2 (2 n + a + b) (2 n + a + b - 2)},	& \text{if $n-m=0$ and $n \geq 0$},\\
    \frac{(n + 1) (n + a + b - 1)}{(2 n + a + b) (2 n + a + b - 1)},	& \text{if $n-m=-1$ and $n \geq 0$}.\\
\end{cases}
\]

This expression is non-zero for $-1 \leq \nIdxCmp_i - \mIdxCmp_i \leq 1$. Furthermore, the form of the integral for $j=i-1$ depends on this difference, or rather the difference between $\nTot_j$ and $\mTot_j$. Depending on the difference $\nTot_j - \mTot_j$ we have to consider the integrals
\begin{align}
	-1 & \; : \; \frac{1}{\uJacLength_{\mIdxCmp_j}^{(\alpha, \beta+2)}} \int_0^1 \uJacobi_{\nIdxCmp_j}^{(\alpha,\beta)}(\xCmpTr) \uJacobi_{\mIdxCmp_j}^{(\alpha,\beta+2)}(\xCmpTr) w_{\alpha,\beta+2}(\xCmpTr) d\xCmpTr,\label{eq_first_up}\\
	0 & \; : \; \frac{1}{\uJacLength_{\mIdxCmp_j}^{(\alpha, \beta)}} \int_0^1 \uJacobi_{\nIdxCmp_j}^{(\alpha,\beta)}(\xCmpTr) \uJacobi_{\mIdxCmp_j}^{(\alpha,\beta)}(\xCmpTr) (1-\xCmpTr) w_{\alpha,\beta}(\xCmpTr) d\xCmpTr,\label{eq_first_stay}\\
	+1 & \; : \; \frac{1}{\uJacLength_{\mIdxCmp_j}^{(\alpha, \beta-2)}} \int_0^1 \uJacobi_{\nIdxCmp_j}^{(\alpha,\beta)}(\xCmpTr) \uJacobi_{\mIdxCmp_j}^{(\alpha,\beta-2)}(\xCmpTr) w_{\alpha,\beta}(\xCmpTr), d\xCmpTr,\label{eq_first_down}
\end{align}
with $\alpha = \mutCmp_j$ and $\beta = \mutTot_j + 2\nTot_j$. In expression~\eqref{eq_first_down} we have to assume $\beta>2$, which is equivalent to $\nTot_j \geq 1$. This holds true, because if $\nTot_j =0$, this case would not have to be considered.

Applying relation~\eqref{eq_increase_beta} twice to the polynomial $\uJacobi_{\nIdxCmp_j}^{(\alpha,\beta)}(\xCmpTr)$ in equation~\eqref{eq_first_up} and using orthogonality yields
\begin{equation}
	H^{(\alpha, \beta)}_{\nIdxCmp_j,\mIdxCmp_j} \delta_{\nIdxCmp_j,\mIdxCmp_j} + H^{(\alpha, \beta)}_{\nIdxCmp_j,\mIdxCmp_j} \delta_{\nIdxCmp_j -1,\mIdxCmp_j} + H^{(\alpha, \beta)}_{\nIdxCmp_j,\mIdxCmp_j} \delta_{ \nIdxCmp_j -2,\mIdxCmp_j},
\end{equation}
for some constants $H^{(\alpha, \beta)}_{n,m}$. Here $H^{(\alpha, \beta)}_{0,-1} =  H^{(\alpha, \beta)}_{0,-2} = H^{(\alpha, \beta)}_{1,-1} = 0$. Thus, in the case $\mTot_j = \nTot_j + 1$, the expression for $j$ is non-zero for $\mIdxCmp_j = \nIdxCmp_j, \nIdxCmp_j - 1,$ and $\nIdxCmp_j -2$. Furthermore, relation~\eqref{eq:R_recurrence} can be applied to the term $\uJacobi_{\nIdxCmp_j}^{(\alpha,\beta)}(\xCmpTr) (1-\xCmpTr)$, together with orthogonality to get
\begin{equation}
	I^{(\alpha,\beta)}_{\nIdxCmp_j,\mIdxCmp_j} \delta_{\nIdxCmp_j+1,\mIdxCmp_j} + I^{(\alpha,\beta)}_{\nIdxCmp_j,\mIdxCmp_j} \delta_{\nIdxCmp_j,\mIdxCmp_j} + I^{(\alpha,\beta)}_{\nIdxCmp_j,\mIdxCmp_j} \delta_{\nIdxCmp_j-1,\mIdxCmp_j},
\end{equation}
for given constants $I^{(\alpha, \beta)}_{n,m}$, with $I^{(\alpha,\beta)}_{-1,0} = I^{(\alpha,\beta)}_{0,-1} = 0$. In the case $\mTot_j = \nTot_j$, the expression is non-zero for $\mIdxCmp_j = \nIdxCmp_j - 1, \nIdxCmp_j ,$ and $\nIdxCmp_j + 1$. Finally, applying relation~\eqref{eq_increase_beta} to the term $\uJacobi_{\mIdxCmp_j}^{(\alpha,\beta-2)}(\xCmpTr)$ in expression~\eqref{eq_first_down} combined with orthogonality yields
\begin{equation}
	J^{(\alpha, \beta)}_{\nIdxCmp_j,\mIdxCmp_j} \delta_{\nIdxCmp_j,\mIdxCmp_j} + J^{(\alpha, \beta)}_{\nIdxCmp_j,\mIdxCmp_j} \delta_{\nIdxCmp_j +1,\mIdxCmp_j} + J^{(\alpha, \beta)}_{\nIdxCmp_j,\mIdxCmp_j} \delta_{\nIdxCmp_j +2,\mIdxCmp_j},
\end{equation}
for given constants $J^{(\alpha, \beta)}_{n,m}$. Again $J^{(\alpha, \beta)}_{-1,0} =  J^{(\alpha, \beta)}_{-2,0} = J^{(\alpha, \beta)}_{-1,1} = 0$. Thus this expression is non-zero for $\mIdxCmp_j = \nIdxCmp_j, \nIdxCmp_j + 1,$ and $\nIdxCmp_j + 2$. 
The constants $H_{n,m}^{(a,b)}, I_{n,m}^{(a,b)}, J_{n,m}^{(a,b)}$ are given by
\begin{align*}
	H_{n,m}^{(a,b)} & = \left\{\begin{array}{ll}
    	\frac{(n + a + b - 1)(n + a + b)}{(2n + a + b - 1)(2n + a + b)},	& \text{if $m-n=0$ and $n \geq 0$},\\
    	-\frac{2a}{a + b + 2},	& \text{if $m-n=-1$ and $n = 1$}\\
    	-\frac{2(n + a -1)(n + a + b - 1)}{(2n + a + b - 2)(2n + a + b)},	& \text{if $m-n=-1$ and $n > 1$},\\
    	\frac{(n + a - 2)(n + a - 1)}{(2n + a + b -2)(2n + a + b - 1)},	& \text{if $m-n=-2$ and $n > 1$},
	\end{array}\right.\\
	I_{n,m}^{(a,b)} &= \begin{cases}
    	- \frac{1}{a+b},	& \text{if $m-n=1$ and $n = 0$},\\
    	- \frac{(n + 1)(n + a + b - 1)}{(2n + a + b - 1)(2n + a + b)},	& \text{if $m-n=1$ and $n > 0$},\\
    	\frac{b}{a + b},	& \text{if $m-n=0$ and $n = 0$},\\
    	\frac{b^2 + a(b + 2)}{(a + b)(a + b + 2)}, & \text{if $m-n=0$ and $n = 1$},\\
    	\frac{b^2 + 2n(n + a - 1) + b(2n + a - 2)}{(2n + a + b - 2)(2n + a + b)}, & \text{if $m-n=0$ and $n > 1$},\\
    	- \frac{a b}{(a + b) (a + b + 1)},	& \text{if $m-n=-1$ and $n = 1$}\\
    	- \frac{(n + a - 1)(n + b - 1)}{(2n + a + b - 2)(2n + a + b - 1)},	& \text{if $m-n=-1$ and $n > 1$},\\
	\end{cases}\\
	J_{n,m}^{(a,b)} &= \begin{cases}
     	\frac{(b - 1) (b - 2)}{(a + b - 1) (a + b - 2)},	& \text{if $m-n=0$ and $n = 0$},\\
     	\frac{(n + b - 2)(n + b - 1)}{(2n + a + b - 2)(2n + a + b -1)},	& \text{if $m-n=0$ and $n > 0$},\\
     	-\frac{2(n + 1)(n + b - 1)}{(2n + a + b - 2)(2n + a + b)},	& \text{if $m-n=1$ and $n \geq 0$},\\
     	\frac{(n + 1)(n + 2)}{(2n + a + b - 1)(2n + a + b)},	& \text{if $m-n=2$ and $n \geq 0$}.\\
	\end{cases}      
\end{align*}

Now considering all three possible values for $\nTot_j - \mTot_j$, and all possible implications for the difference $\nIdxCmp_j - \mIdxCmp_j$, it can be shown that $1 \leq \nTot_{j-1} - \mTot_{j-1} \leq 1$ has to hold as well. Using induction shows that $1 \leq \nTot_{j} - \mTot_{j} \leq 1$ holds for all $j<i$. Thus for all $j<i$ the same integrals~\eqref{eq_first_up},~\eqref{eq_first_stay}, and~\eqref{eq_first_down}, with adjusted parameters $\alpha = \mutCmp_j$ and $\beta = \mutTot_j + 2\nTot_j$, have to be considered.

Combining these results shows that for fixed $i$ and $\nIdx$ the polynomials with a non-zero contribution to the recurrence relation for $\xCmp_i \jacobi^\mutVec_\nIdx (\xVec)$ are exactly those with indices from the set
\begin{equation}
	\M_i (\nIdx) := \Big\{ \mIdx \in \N_0^{\numAlleles-1} \Big\vert m_j \geq 0 \,\forall j, \mTot_j = \nTot_j \,\forall j>i, | \mTot_j - \nTot_j| \leq 1 \,\forall j \leq i\Big\},
\end{equation}
defined in~\eqref{eq_index_set}.
Thus,
\begin{equation}
	\xCmp_i \jacobi^\mutVec_\nIdx (\xVec) = \sum_{\mIdx \in \M_i(\nIdx)} r^{(\mutVec,i)}_{\nIdx,\mIdx} \jacobi^\mutVec_\mIdx (\xVec),
\end{equation}
where the coefficients $r^{(\mutVec,i)}_{\nIdx,\mIdx}$ are given by
\begin{equation}\label{eq_multivariate_coefficients}
	r^{(\mutVec,i)}_{\nIdx,\mIdx} = G^{(\mutCmp_i, \mutTot_i + 2\nTot_i)}_{\nIdxCmp_i,\mIdxCmp_i} \prod_{j<i} \,
	\begin{cases} 
			   \displaystyle  H^{(\mutCmp_j, \mutTot_j + 2\nTot_j)}_{\nIdxCmp_j,\mIdxCmp_j},	& \text{if $\nTot_j-\mTot_j = -1$},\\
			   \displaystyle  I^{(\mutCmp_j, \mutTot_j + 2\nTot_j)}_{\nIdxCmp_j,\mIdxCmp_j},	& \text{if $\nTot_j-\mTot_j = 0$},\\
			   \displaystyle  J^{(\mutCmp_j, \mutTot_j + 2\nTot_j)}_{\nIdxCmp_j,\mIdxCmp_j},	& \text{if $\nTot_j-\mTot_j= +1$}.
			\end{cases}
\end{equation}
\end{proof}

\begin{proof}[Proof of Lemma~\ref{lem_eigen_neutral}]
Using the coordinate transformation introduced in Appendix~\ref{app_coord_transform}, and applying $\genNeutral$, given in equation~\eqref{eq_transformed_generator}, to $\jacobiTr_{\nIdx}^{\mutVec}(\xVecTr)$ from equation~\eqref{eq_transformed_jacobis} yields
\begin{equation}\label{eq_gen_to_poly}
	\begin{split}
		\genNeutral \jacobiTr_\nIdx^{\mutVec}(\xVecTr) & = \frac{1}{2}\sum_{i=1}^{\numAlleles-1}\frac{1}{\prod_{k<i}(1-\xCmpTr_k)}\prod_{j=1,j\neq i}^{\numAlleles-1}\uJacobi_{\nIdxCmp_j}^{(\mutCmp_j,\mutTot_j+2\nTot_j)}(\xCmpTr_j)(1-\xCmpTr_j)^{\nTot_j}\\
			& \qquad \qquad \times\left(\xCmpTr_i(1-\xCmpTr_i)\frac{\partial^2}{\partial \xCmpTr_i^2}\left\{\uJacobi_{\nIdxCmp_i}^{(\mutCmp_i,\mutTot_i+2\nTot_i)}(\xCmpTr_i)(1-\xCmpTr_i)^{\nTot_i}\right\}\right.\\
			& \qquad \qquad \qquad\left.+(\mutCmp_i-\mutTot_{i-1}\xCmpTr_i)\frac{\partial}{\partial \xCmpTr_i}\left\{\uJacobi_{\nIdxCmp_i}^{(\mutCmp_i,\mutTot_i+2\nTot_i)}(\xCmpTr_i)(1-\xCmpTr_i)^{\nTot_i}\right\}\right).
	\end{split}
\end{equation}
Employing equation~\eqref{eq_R_ode}, one can show that the terms in the brackets on the right hand side of equation~\eqref{eq_gen_to_poly} reduce to \begin{equation}
(1-\xCmpTr_i)^{\nTot_i}\uJacobi_{\nIdxCmp_i}^{(\mutCmp_i,\mutTot_i+2\nTot_i)}(\xCmpTr_i)(-\nTot_{i-1}(\nTot_{i-1}-1+\mutTot_{i-1})+\frac{1}{1-\xCmpTr_i}\nTot_i(\nTot_i-1+\mutTot_i)),
\end{equation} and substitution yields 
\begin{eqnarray*}
\genNeutral \jacobiTr_\nIdx^{\mutVec}(\xVecTr) & = & \frac{1}{2}\jacobiTr_\nIdx^{\mutVec}(\xVecTr) \left(-\sum_{i=1}^{\numAlleles-1}\frac{1}{\prod_{k<i}(1-\xCmpTr_k)}\nTot_{i-1}(\nTot_{i-1}-1+\mutTot_{i-1})\right.\nonumber \\
 & & \qquad\qquad\qquad\left.+\sum_{i=2}^{\numAlleles}\frac{1}{\prod_{k<i}(1-\xCmpTr_k)}\nTot_{i-1}(\nTot_{i-1}-1+\mutTot_{i-1})\right) \\
  & = & -\lambda^\mutVec_{\vert\nIdx\vert} \jacobiTr_{\nIdx}^{\mutVec}(\xVecTr)
\end{eqnarray*}
with $\lambda^\mutVec_{\vert\nIdx\vert}=\frac{1}{2}|\nIdx|(|\nIdx|-1+|\mutVec|)$, since $\mutTot_0 = |\mutVec|$, $\nTot_0 = |\nIdx|$, and $\nTot_{K-1} = 0$.
\end{proof}

\section{Change of coordinates}
\label{app_coord_transform}

Working with the multivariate Jacobi polynomials and the neutral diffusion, it is convenient, for some derivations, to transform the equations to a different coordinate system. This transformation maps the simplex $\simplex_{\numAlleles-1}$ to the $\numAlleles-1$-dimensional unit cube $[0,1]^{\numAlleles-1}$. It is implicitly used in \citet[Section~3]{Griffiths2011}, but more explicitly introduced and used as a transformation in \citet{Baxter2007}. The vector $\xVecTr(\xVec) = (\xCmpTr_1(\xVec), \ldots, \xCmpTr_{\numAlleles-1}(\xVec))$ is obtained from the vector of population frequencies $\xVec$ via the transformation
\begin{equation}\label{eq_xi_transform}
	\xCmpTr_i = \frac{\xCmp_i}{1-\sum_{j<i}\xCmp_j}
\end{equation}
for $1 \leq i \leq \numAlleles-1$. The inverse of this transformation is given by
\begin{equation}\label{eq_x_inverse}
	\xCmp_i = \xCmpTr_i \prod_{j<i} (1 - \xCmpTr_j)
\end{equation}
for $1 \leq i \leq \numAlleles-1$. The inverse relation can be derived by noting that $1-\sum_{j<i}\xCmp_j = \prod_{j<i} (1 - \xCmpTr_j)$ holds.

Definition~\ref{def_multi_jacobi} yields immediately that the multivariate Jacobi polynomials $\jacobi_{\nIdx}^{\mutVec}(\xVec)$ take the form
\begin{equation}\label{eq_transformed_jacobis}
	\jacobiTr_{\nIdx}^{\mutVec}(\xVecTr) = \jacobi_{\nIdx}^{\mutVec}\big(\xVec(\xVecTr)\big) = \prod_{j=1}^{\numAlleles-1} \uJacobi_{\nIdxCmp_j}^{(\mutCmp_j,\mutTot_j + 2\nTot_j)}(\xCmpTr_j) (1-\xCmpTr_j)^{\nTot_j}
\end{equation}
in the transformed coordinates. The neutral diffusion generator $\genNeutral$ in the transformed coordinate system is given by the following lemma.
\begin{lemma}
Using variables in the new coordinate system, the backward generator of the diffusion under neutrality $\genNeutral$ can be written as
\begin{equation}\label{eq_transformed_generator}
\genNeutral \tilde{f}(\xVecTr) = \frac{1}{2} \sum_{i=1}^{\numAlleles-1} \diffTermTr_{i,i}(\xVecTr) \frac{\partial^{2}}{\partial \xCmpTr_{i}^{2}} \tilde{f}(\xVecTr) + \sum_{i=1}^{\numAlleles-1} \driftTermTr_{i}(\xVecTr) \frac{\partial}{\partial \xCmpTr_{i}} \tilde{f}(\xVecTr),
\end{equation}
with \begin{equation}
\diffTermTr_{i,j}(\xVecTr) = \delta_{i,j} \Bigg( \frac{\xCmpTr_{i}(1-\xCmpTr_{i})}{\prod_{k<i}(1-\xCmpTr_{k})} \Bigg)
\end{equation}
and \begin{equation}
\driftTermTr_{i}(\xVecTr) = \frac{1}{2} \frac{\mutCmp_{i}-\mutTot_{i-1}\xCmpTr_{i}}{\prod_{k<i}(1-\xCmpTr_{k})}.
\end{equation}
\end{lemma}
The proof of this lemma is paraphrased in Appendix~B of \cite{Baxter2007}. The transformation diagonalizes the operator by removing all the mixed second order partial derivatives.

\section{Coefficients of the polynomial $Q(\xVec;\selMatrix,\mutVec)$}
\label{app_poly_coeff}

\begin{eqnarray*}
q & = & \frac{1}{2}(\sum_{j=1}^{\numAlleles}\mutCmp_{j}\selCmp_{\numAlleles,j}-|\mutVec|\selCmp_{\numAlleles,\numAlleles})\text{ when ${\bf i}=\varnothing$},\\
q(i_{1}) & = & \frac{1}{2}(\sum_{j=1}^{\numAlleles}\mutCmp_{j}(\selCmp_{i_{1},j}-\selCmp_{t,\numAlleles})+\selCmp_{i_{1},\numAlleles}^{2}+\selCmp_{\numAlleles,\numAlleles}^{2}-2\selCmp_{\numAlleles,\numAlleles}\selCmp_{i_{1},\numAlleles}\\
 &  & -2(1+|\mutVec|)\selCmp_{i_{1},\numAlleles}+(1+2|\mutVec|)\selCmp_{\numAlleles,\numAlleles}+\selCmp_{i_{1},i_{1}}),\\
q(i_{1},i_{2}) & = & \frac{1}{2}(2\selCmp_{i_{1},\numAlleles}\selCmp_{i_{1},i_{2}}-3\selCmp_{i_{1},\numAlleles}\selCmp_{i_{2},\numAlleles}+8\selCmp_{i_{2},\numAlleles}\selCmp_{\numAlleles,\numAlleles}-2\selCmp_{\numAlleles,\numAlleles}\selCmp_{i_{1},i_{2}}-2\selCmp_{i_{1},\numAlleles}^{2}-3\selCmp_{\numAlleles,\numAlleles}^{2}\\
 &  & -(1+|\mutVec|)(\selCmp_{i_{1},i_{2}}+\selCmp_{\numAlleles,\numAlleles}-2\selCmp_{i_{2},\numAlleles})),\\
q(i_{1},i_{2},i_{3}) & = & \frac{1}{2}((\selCmp_{i_{1},i_{3}}-\selCmp_{i_{1},\numAlleles})(\selCmp_{i_{1},i_{2}}-\selCmp_{i_{1},\numAlleles})-(\selCmp_{i_{3},\numAlleles}-\selCmp_{\numAlleles,\numAlleles})(\selCmp_{i_{2},\numAlleles}-\selCmp_{\numAlleles,\numAlleles})\\
 &  & -4(\selCmp_{i_{2},i_{3}}+\selCmp_{\numAlleles,\numAlleles}-2\selCmp_{i_{3},\numAlleles})(\selCmp_{i_{1},\numAlleles}-\selCmp_{\numAlleles,\numAlleles})),\\
q(i_{1},i_{2},i_{3},i_{4}) & =- & \frac{1}{2}((\selCmp_{i_{1},i_{2}}+\selCmp_{\numAlleles,\numAlleles}-2\selCmp_{i_{2},\numAlleles})(\selCmp_{i_{3},i_{4}}+\selCmp_{\numAlleles,\numAlleles}-2\selCmp_{i_{4},\numAlleles})).
\end{eqnarray*}

 \section{Derivation of equation~\eqref{eq_L_on_H}}
\label{app_derivation_polynomial}

Applying $\genFull $ to $\modJacobi_n(\xVec)$,
\begin{equation}\label{eq:L_on_H}
	\begin{split}
		\genFull  \modJacobi_{\nIdx}(\xVec) & = (\genNeutral + \genSel)(\jacobi_\nIdx^{\mutVec}(\xVec)e^{-\meanFitness(\xVec)/2})\\
			& = e^{-\frac{\meanFitness(\xVec)}{2}}\genNeutral \jacobi_{\nIdx}^{\mutVec}(\xVec)+\jacobi_\nIdx^{\mutVec}(\xVec)\genNeutral e^{-\frac{\meanFitness(\xVec)}{2}}+\sum_{i,j=1}^{\numAlleles-1}\xCmp_{i}(\delta_{i,j}-\xCmp_{j})\frac{\partial}{\partial \xCmp_{i}} \big\{ e^{-\frac{\meanFitness(\xVec)}{2}} \big\} \frac{\partial}{\partial \xCmp_{j}} \big\{ \jacobi_{\nIdx}^{\mutVec}(\xVec) \big\} \\
			& \qquad +\jacobi_\nIdx^{\mutVec}(\xVec)\genSel e^{-\frac{\meanFitness(\xVec)}{2}}+e^{-\frac{\meanFitness(\xVec)}{2}}\genSel \jacobi_{\nIdx}^{\mutVec}(\xVec)\\
			& = -\lambda^\mutVec_\nIdx e^{-\frac{\meanFitness(x)}{2}}\jacobi_\nIdx^{\mutVec}(\xVec)+\jacobi_\nIdx^{\mutVec}(\xVec)\genFull  e^{-\frac{\meanFitness(\xVec)}{2}}\\
			& \qquad +\sum_{i,j=1}^{\numAlleles-1}\xCmp_{i}(\delta_{i,j}-\xCmp_{j})\frac{\partial}{\partial \xCmp_{i}} \big\{ e^{-\frac{\meanFitness(\xVec)}{2}} \big\} \frac{\partial}{\partial \xCmp_{j}} \big\{ \jacobi_{\nIdx}^{\mutVec}(\xVec) \big\} +e^{-\frac{\meanFitness(\xVec)}{2}}\genSel \jacobi_{\nIdx}^{\mutVec}(\xVec).
	\end{split}
\end{equation}
It can be shown that the last two terms in the above expression
sum up to 0. Note that for $1\leq i,j\leq \numAlleles-1,$
\begin{equation}\label{eq_partial_mean_fitness}
	\frac{\partial}{\partial \xCmp_{i}}\meanFitness(\xVec)=2\sum_{k=1}^{\numAlleles}\selCmp_{k,i}\xCmp_{k}-2\sum_{l=1}^{\numAlleles}\selCmp_{l,\numAlleles}\xCmp_{l},
\end{equation}
\begin{equation}\label{eq_partial_partial_mean_fitness}
	\frac{\partial^{2}}{\partial \xCmp_{j}\partial \xCmp_{i}}\meanFitness(\xVec)=2(\selCmp_{i,j}-\selCmp_{j,\numAlleles}-\selCmp_{i,\numAlleles}+\selCmp_{\numAlleles,\numAlleles}).	
\end{equation}
It follows that 
\begin{eqnarray*}
 &  & \sum_{i,j=1}^{\numAlleles-1}\xCmp_{i}(\delta_{i,j}-\xCmp_{j})\frac{\partial}{\partial \xCmp_{i}} \big\{ e^{-\frac{\meanFitness(\xVec)}{2}} \big\} \frac{\partial}{\partial \xCmp_{j}} \big\{ \jacobi_{\nIdx}^{\mutVec}(\xVec)  \big\}  +e^{-\frac{\meanFitness(\xVec)}{2}}\genSel \jacobi_{\nIdx}^{\mutVec}(\xVec)\\
 & = & e^{-\frac{\meanFitness(\xVec)}{2}}\Bigg[ \sum_{i=1}^{\numAlleles-1}\xCmp_{i}\frac{\partial}{\partial \xCmp_{i}} \bigg\{ -\frac{\meanFitness(\xVec)}{2} \bigg\} \frac{\partial}{\partial \xCmp_{i}}  \big\{ \jacobi_{\nIdx}^{\mutVec}(\xVec)  \big\} -\sum_{i,j=1}^{\numAlleles-1}\xCmp_{i}\xCmp_{j}\frac{\partial}{\partial \xCmp_{i}} \bigg\{ -\frac{\meanFitness(\xVec)}{2} \bigg\} \frac{\partial}{\partial \xCmp_{j}} \big\{ \jacobi_{\nIdx}^{\mutVec}(\xVec) \big\}\\
 &  & \qquad \qquad +\sum_{i=1}^{\numAlleles-1}\xCmp_{i}\frac{\partial}{\partial \xCmp_{i}} \big\{ \jacobi_{\nIdx}^{\mutVec}(\xVec) \big\} \sum_{j=1}^{\numAlleles}\selCmp_{i,j}\xCmp_{j}-\meanFitness(\xVec)\sum_{i=1}^{\numAlleles-1}\xCmp_{i}\frac{\partial}{\partial \xCmp_{i}} \big\{ \jacobi_{\nIdx}^{\mutVec}(\xVec) \big\} \Bigg] \\
 & = & e^{-\frac{\meanFitness(\xVec)}{2}} \Bigg[ -\sum_{i=1}^{\numAlleles-1}\xCmp_{i} \bigg( \sum_{k=1}^{\numAlleles}\selCmp_{k,i}\xCmp_{k}-\sum_{l=1}^{\numAlleles}\selCmp_{l,\numAlleles}\xCmp_{l} \bigg) \frac{\partial}{\partial \xCmp_{i}} \big\{ \jacobi_{\nIdx}^{\mutVec}(\xVec) \big\}\\
  &  & \qquad \qquad +\sum_{i,j=1}^{\numAlleles-1}\xCmp_{i}\xCmp_{j} \bigg( \sum_{k=1}^{\numAlleles}\selCmp_{k,i}\xCmp_{k}-\sum_{l=1}^{\numAlleles}\selCmp_{l,\numAlleles}\xCmp_{l} \bigg) \frac{\partial}{\partial \xCmp_{j}} \big\{ \jacobi_{\nIdx}^{\mutVec}(\xVec) \big\}\\
 &  & \qquad \qquad +\sum_{i=1}^{\numAlleles-1}\xCmp_{i}\frac{\partial}{\partial \xCmp_{i}} \big\{ \jacobi_{\nIdx}^{\mutVec}(\xVec) \big\} \sum_{j=1}^{\numAlleles}\selCmp_{i.j}\xCmp_{j}-\meanFitness(\xVec)\sum_{i=1}^{\numAlleles-1}\xCmp_{i}\frac{\partial}{\partial \xCmp_{i}} \big\{ \jacobi_{\nIdx}^{\mutVec}(\xVec) \big\} \Bigg] \\
 & = & e^{-\frac{\meanFitness(\xVec)}{2}} \Bigg[ \sum_{i=1}^{\numAlleles-1}\xCmp_{i}\frac{\partial}{\partial \xCmp_{i}} \big\{ \jacobi_{\nIdx}^{\mutVec}(\xVec)  \big\} \sum_{l=1}^{\numAlleles}\selCmp_{l,\numAlleles}\xCmp_{l}\\
& & \qquad \qquad -\sum_{j=1}^{\numAlleles-1}\xCmp_{j}\frac{\partial}{\partial \xCmp_{j}} \big\{ \jacobi_{\nIdx}^{\mutVec}(\xVec) \big\} \bigg( \sum_{l=1}^{\numAlleles}\selCmp_{l,\numAlleles}\xCmp_{\numAlleles}\xCmp_{l}+\sum_{i=1}^{\numAlleles-1}\xCmp_{i}\sum_{l=1}^{\numAlleles}\selCmp_{l,\numAlleles}\xCmp_{l} \bigg) \Bigg] \\
 & = & 0,\end{eqnarray*}
where we used equation~\eqref{eq_partial_mean_fitness} for the second equality and $\sum_{i=1}^{\numAlleles}\xCmp_{i}=1$ for the last equality.

Further, using equation~\eqref{eq_partial_mean_fitness} and~\eqref{eq_partial_partial_mean_fitness} one can show that
\begin{eqnarray*}
\genFull  e^{-\frac{\meanFitness(\xVec)}{2}} & = & \frac{1}{2}e^{-\frac{\meanFitness(\xVec)}{2}}\left(-\sum_{i=1}^{\numAlleles}\xCmp_{i}\selCmp_{i}^{2}(\xVec)-\sum_{i=1}^{\numAlleles}\xCmp_{i}\selCmp_{ii}+(1+|\mutVec|)\meanFitness(\xVec)+\meanFitness(\xVec)^{2}-\sum_{i=1}^{\numAlleles}\mutCmp_{i}\selCmp_{i}(\xVec)\right)\\
 & = & -e^{-\frac{\meanFitness(\xVec)}{2}}Q(\xVec;\selMatrix,\mutVec),\end{eqnarray*}
where $Q$ takes the form \eqref{eq_Q}, that is $Q(\xVec;\selMatrix,\mutVec) = \sum_{\iIdx\in\iSet}q(\iIdx)\xVec_{\iIdx}$, with the constants $q(\iIdx)$ given in Appendix~\ref{app_poly_coeff}.

\bibliographystyle{myplainnat}
\bibliography{bib}

\end{document}

%% file: macros.tex
\def\bfmath#1{\mathchoice
        {\mbox{\boldmath$#1$}}%
        {\mbox{\boldmath$#1$}}%
        {\mbox{\boldmath$\scriptstyle#1$}}%
        {\mbox{\boldmath$\scriptscriptstyle#1$}}}%

\newcommand{\numAlleles}{\ensuremath{K}}
\newcommand{\popSize}{\ensuremath{N}}
\newcommand{\mutVec}{{\ensuremath{\boldsymbol{\theta}}}}
\newcommand{\mutCmp}{\ensuremath{\theta}}
\newcommand{\mutTot}{\ensuremath{\Theta}}
\newcommand{\selMatrix}{\ensuremath{\boldsymbol{\sigma}}}
\newcommand{\selCmp}{\ensuremath{\sigma}}
\newcommand{\meanFitness}{\bar\sigma}

\newcommand{\iIdx}{\ensuremath{\mathbf{i}}}
\newcommand{\iCmp}{\ensuremath{i}}
\newcommand{\iSet}{\ensuremath{\mathcal{I}}}

\newcommand{\xVec}{\ensuremath{\mathbf{x}}}
\newcommand{\xCmp}{\ensuremath{x}}

\newcommand{\xVecTr}{\ensuremath{\boldsymbol{\xi}}}
\newcommand{\xCmpTr}{\ensuremath{\xi}}
\newcommand{\yVec}{\ensuremath{\mathbf{y}}}
\newcommand{\yCmp}{\ensuremath{y}}
\newcommand{\simplex}{\ensuremath{\Delta}}

\newcommand{\uJacobi}{\ensuremath{R}}
\newcommand{\uJacLength}{\ensuremath{c}}
\newcommand{\jacobi}{\ensuremath{P}}
\newcommand{\jacobiTr}{\ensuremath{\tilde{P}}}
\newcommand{\eVecNeutr}{\ensuremath{\lambda}}

\newcommand{\jacLength}{\ensuremath{C}}
\newcommand{\polyMatrix}{\ensuremath{\mathcal{G}}}
\newcommand{\modJacobi}{\ensuremath{S^{\mutVec}}}
\newcommand{\kIdx}{\ensuremath{\mathbf{k}}}
\newcommand{\lIdx}{\ensuremath{\mathbf{l}}}
\newcommand{\nIdx}{\ensuremath{\mathbf{n}}}
\newcommand{\nIdxCmp}{\ensuremath{n}}
\newcommand{\nTot}{\ensuremath{N}}
\newcommand{\mIdx}{\ensuremath{\mathbf{m}}}
\newcommand{\mIdxCmp}{\ensuremath{m}}
\newcommand{\mTot}{\ensuremath{M}}
\newcommand{\statNeutr}{\ensuremath{\Pi_0}}

\newcommand{\eVal}{\ensuremath{\Lambda}}
\newcommand{\eVec}{\ensuremath{\mathbf{u}}}
\newcommand{\uCoeff}{\ensuremath{u}}
\newcommand{\eFun}{\ensuremath{B}}

\newcommand{\eIdx}{\ensuremath{n}}
\newcommand{\eIdxM}{\ensuremath{m}}

\newcommand{\genFull}{\ensuremath{\mathscr{L}}}
\newcommand{\genNeutral}{\ensuremath{\mathscr{L}_0}}
\newcommand{\genSel}{\ensuremath{\mathscr{L}_\sigma}}
\newcommand{\alleleFreqRV}{\ensuremath{\mathbf{X}}}
\newcommand{\diffTerm}{\ensuremath{b}}
\newcommand{\driftTerm}{\ensuremath{a}}
\newcommand{\diffTermTr}{\ensuremath{\tilde{b}}}
\newcommand{\driftTermTr}{\ensuremath{\tilde{a}}}

\newcommand{\stat}{\ensuremath{\Pi}}
\newcommand{\normConst}{\ensuremath{C_\stat}}

\newcommand{\matrixCutoff}{\ensuremath{D}}
\newcommand{\smallerIndices}{\ensuremath{\mathcal{U}}}

\newcommand{\sref}[1]{Section~\ref{#1}}

\newcommand{\fref}[1]{Figure~\ref{#1}}
\newcommand{\cref}[1]{Chapter~\ref{#1}}

\def\cutoff{\matrixCutoff}
\def\supN{{[\cutoff]}}

\def\bfa{{\bfmath{a}}}

\def\bfn{\bfmath{n}}

\def\bfw{\bfmath{w}}